\documentclass[a4paper,11pt]{article}

\usepackage[utf8]{inputenc}
\usepackage{graphicx}
\usepackage{amsmath}
\usepackage{amsfonts}
\usepackage{amssymb}
\usepackage{amsthm}
\usepackage{color}
\usepackage{ulem}
\usepackage{pdfsync}
\usepackage[mathscr]{euscript}

\usepackage[text={15cm,24cm},centering]{geometry}
\usepackage{subcaption}

\newcommand{\bmat}{\begin{matrix}}
\newcommand{\emat}{\end{matrix}}	
\newcommand{\bpm}{\begin{pmatrix}}
\newcommand{\epm}{\end{pmatrix}}

\theoremstyle{definition}
\newtheorem{theorem}{\bf Theorem}

\newcommand{\real}{\mathbb{R}}
\newcommand{\mb}{\mathbf}

\newcommand{\ls}{\leqslant}
\newcommand{\gs}{\geqslant}

\newcommand{\demi}{\frac{1}{2}}

\newcommand{\hh}{\hspace{1cm}}
\newcommand{\hd}{\hspace{5mm}}

\newcommand{\p}{\partial}

\newcommand{\er}{\eqref}

\title{Feedforward boundary control\\ of 2$\times$2 nonlinear hyperbolic systems \\ with application to Saint-Venant equations}

\author{Georges Bastin\thanks{Department of Mathematical Engineering, ICTEAM, UCLouvain, 4, Avenue G. Lemaitre, 1348 Louvain-La-Neuve, Belgium. georges.bastin@uclouvain.be} \,,\, Jean-Michel Coron\thanks{Laboratoire Jacques-Louis Lions, Sorbonne Université, Université de Paris, CNRS, INRIA, Laboratoire Jacques-Louis Lions, équipe Cage, Paris, France. coron@ann.jussieu.fr}\,\, and Amaury Hayat\thanks{Department of Mathematics, Rutgers University, Camden, USA; and CERMICS, Ecole des Ponts Paristech, Champs-sur-Marne, France}}

\date\today

\begin{document}

\maketitle

\begin{abstract}
\noindent Because they represent physical systems with propagation delays, hyperbolic systems are well suited for feedforward control. This is especially true when the delay between a disturbance and the output is larger than the control delay. In this paper, we address the design of feedforward controllers for a general class of $2 \times 2$ hyperbolic systems with a single disturbance input located at one boundary and a single control actuation at the other boundary. The goal is to design a feedforward control that makes the system output insensitive to the measured disturbance input. We show that, for this class of systems, there exists an efficient ideal feedforward controller which is causal and stable. The problem is first stated and studied in the frequency domain for a simple linear system. Then, our main contribution is to show how the theory can be extended, in the time domain, to general nonlinear hyperbolic systems. The method is illustrated with an application to the control of an open channel represented by Saint-Venant equations where the objective is to make the output water level insensitive to the variations of the input flow rate. Finally, we address a more complex application to a cascade of pools where a blind application of perfect feedforward control can lead to detrimental oscillations. A pragmatic way of modifying the control law to solve this problem is proposed and validated with a simulation experiment.\\

\end{abstract}

\section{Introduction}

Feedforward control is a technique which is of interest when the system to be controlled is subject to a significant input disturbance that can be measured and compensated before it affects the system output. This control technique is used in many control engineering applications, especially in the industrial process sector (e.g. \cite[Chapter 15]{SebEdgMel11}). In ideal situations, feedforward control is an open-loop technique which is theoretically able to achieve perfect control by anticipating adequately the effect of the perturbations. This is in contrast with closed-loop feedback control where corrective actions take place necessarily only after the effect of the disturbances has been detected at the output. However ideal feedforward controllers, which are based on some sort of process model inversion, may not be physically realizable because they can be non causal and/or unstable (e.g. \cite{Gla19}). In such situations, it is common practice to design approximate low order realizable feedforward controllers, possibly in combination with feedback control (e.g. \cite{GuzHag11}, \cite{HasHag12}). Furthermore, it is also well known that, in some instances, feedforward control may be a simple and low cost way of avoiding loss of stability due to actuator saturations in feedback loops (e.g. \cite{Deu17a}, \cite{HovBit09}). For finite dimensional linear systems, the theory of feedforward control is well established and the basics can be found, for instance, in the classical textbooks \cite{MorZaf89} and \cite{SebEdgMel11}. 

In this paper, we are concerned with the application of the feedforward technique to the boundary control of 1-D hyperbolic systems. Our purpose is to address the design of feedforward controllers for a general class of $2 \times 2$ hyperbolic systems with a single disturbance input located at one boundary and a single control actuation at the other boundary. This class of systems includes many potential interesting applications, including those that are listed in the book \cite[Chapter 1]{BasCor14} for example.

Hyperbolic systems generally represent physical phenomena with important propagation delays. For that reason, they are particularly suitable for the implementation of feedforward control, especially when the input/output disturbance delay is larger than the control delay. In that case, as we shall see in this paper, it is indeed possible to design efficient ideal feedforward controllers that are causal and stable, and significantly improve the system performance.

For hyperbolic systems with \textbf{unmeasurable} disturbance inputs produced by a so-called exogenous ``signal system'', the asymptotic \textit{closed-loop} rejection of disturbances by feedback of measurable outputs was extensively considered in the literature in the recent years, especially in the backstepping framework. Significant contributions on this topic have been published, among others, by Ole Morten Aamo (e.g.\cite{Aam13}, \cite{AnfAam17b}) and Joachim Deutsher (e.g.\cite{Deu17b}, \cite{DeuGab19}) and their collaborators. It is worth noting that the viewpoint adopted in the present paper is rather different, since we consider  systems having an arbitrary and \textbf{measurable} disturbance input for which it is desired to design an \textit{open-loop} control of an output variable which is not measured. We show that, for the class of systems considered in this paper, there exists an ideal causal controller that achieves perfect control with stability.

The feedforward control problem considered in this paper is defined and presented in the next Section 2. The physical system to be controlled is described by a $2 \times 2$ quasi-linear hyperbolic system with a density $H$ and a flow density $Q$ as state variables. The goal is to design a \textit{feedforward control law} that makes the system output insensitive to the measured disturbance input. The overall control system is represented by the simple block diagram shown in Figure \ref{configFFsystem}.

In Section 3, we first examine the simplest linear case, i.e. a physical system of two linear conservation laws with constant characteristic velocities. The reason for beginning in this way is that it allows an explicit and complete mathematical analysis of the feedforward control design in the frequency domain. Furthermore it also allows to derive  an expression of the control law in the time domain that can then be used to justify the feedforward control design in the general nonlinear case.

Section 4 is then devoted to a theoretical analysis of the feedforward control design in the general nonlinear case. It is first shown that there exists an ideal causal feedforward dynamic controller that achieves perfect control. In a second step, sufficient  conditions are given under which the controller, in addition to being causal, ensures the stability of the overall control system.

Applications are then presented. First, in Section 5, generalizing the previous results of  \cite{BasCordAn05} and  \cite[Section 9]{LitFro09c}, it is shown how the theory can be directly applied to the control of an open channel whose dynamics are represented by the Saint-Venant equations. The control action is provided by a hydraulic gate  at the downstream side of the channel. The control objective is to make the output water level insensitive to the variations of the input flow rate at the upstream side. The method is illustrated with a realistic simulation experiment.

Then in Section 6, we address the more complex application of the control of a long canal made up of a cascade of a large number of successive pools, as it is the case in navigable rivers for instance. It is then shown that, in this case, a blind application of perfect feedforward controllers leads to oscillations in the downstream direction that can be detrimental in practice. A pragmatic and efficient way of modifying the control design to solve this problem is proposed and validated with simulation results. 

\section{The feedforward control problem}

Let us consider a physical system represented by a general 2$\times$2 nonlinear hyperbolic system of the form
\begin{gather}
H_t + Q_x = 0, \label{mod1}\\[0.5em]
Q_t + (f(H,Q))_x + g(H,Q) = 0, \label{mod2}
\end{gather}
where~:
\begin{itemize}
\item $t$ and $x$ are the two independent variables: a time variable $t \in [0, + \infty)$ and a space variable $x \in [0, L]$ on a finite interval;
\item $(H,Q) : [0, + \infty) \times [0, L] \rightarrow \real^2$ is the vector of the two dependent variables (i.e. $H(t,x)$ and $Q(t,x)$ are the two states of the system);
\item $f:\real^2 \rightarrow \real$ and $g:\real^2 \rightarrow \real$ are sufficiently smooth functions.
\end{itemize}
%\vskip 1em
The first equation \eqref{mod1} can be interpreted as a mass conservation law with $H$ the density and $Q$ the flux density. The second equation \eqref{mod2} can then be interpreted as a momentum balance law.

We are concerned with the solutions of the Cauchy problem for the system \eqref{mod1}--\eqref{mod2} over $[0, + \infty) \times [0, L]$ under an initial condition:
\begin{equation} \label{icmod1}
 \big(H(0,x), Q(0,x) \big) \hh x \in [0,L]
\end{equation}
and two local boundary conditions of the form:
\begin{gather}
\alpha(H(t,0), Q(t,0)) = D(t), \hd t\in [0, + \infty) \label{bc1},\\[0.5em]
\beta( H(t,L), Q(t,L)) = U(t), \hd t\in [0, + \infty)  \label{bc2},
\end{gather}
where $\alpha:\real^2 \rightarrow \real$, $\beta:\real^2 \rightarrow \real$ are sufficiently smooth functions. 

At the left boundary (i.e. $x=0$), the function $D(t)$ is supposed to be a bounded measurable time-varying disturbance. At the right boundary (i.e. $x=L$), $U(t)$ is a control function that can be freely selected by the operator.

The control objective is to keep the output density $H(t,L)$ insensitive to the variations of the disturbance $D(t)$. More precisely, we consider the problem of finding a {\textit {feedforward control law}} $U(t)$, function of the measured disturbance $D(t)$, such that the output density $H(t,L)$ is identically equal to a desired value $H^*_L$ (called `set point'), i.e. $H(t,L) \equiv H^*_L$ $\forall t$. Equivalently it is required that the output function  $Y(t) = H(t,L) - H^*_L$ is identically zero along the solutions of the Cauchy problem. In less technical terms, we want a control law which exactly cancels the influence of the left boundary disturbance $D(t)$ on the right boundary state $H(t,L)$. The overall control system configuration is illustrated in Fig.\ref{configFFsystem}. As we can see in this figure, we thus have a series interconnection of two dynamical systems : the ``Physical System'' with output $Y(t)$ and inputs $D(t), U(t),$ and the feedforward controller with output $U(t)$ and input $D(t)$. 

\begin{figure}[hbt]
  \includegraphics[width=0.90\textwidth]{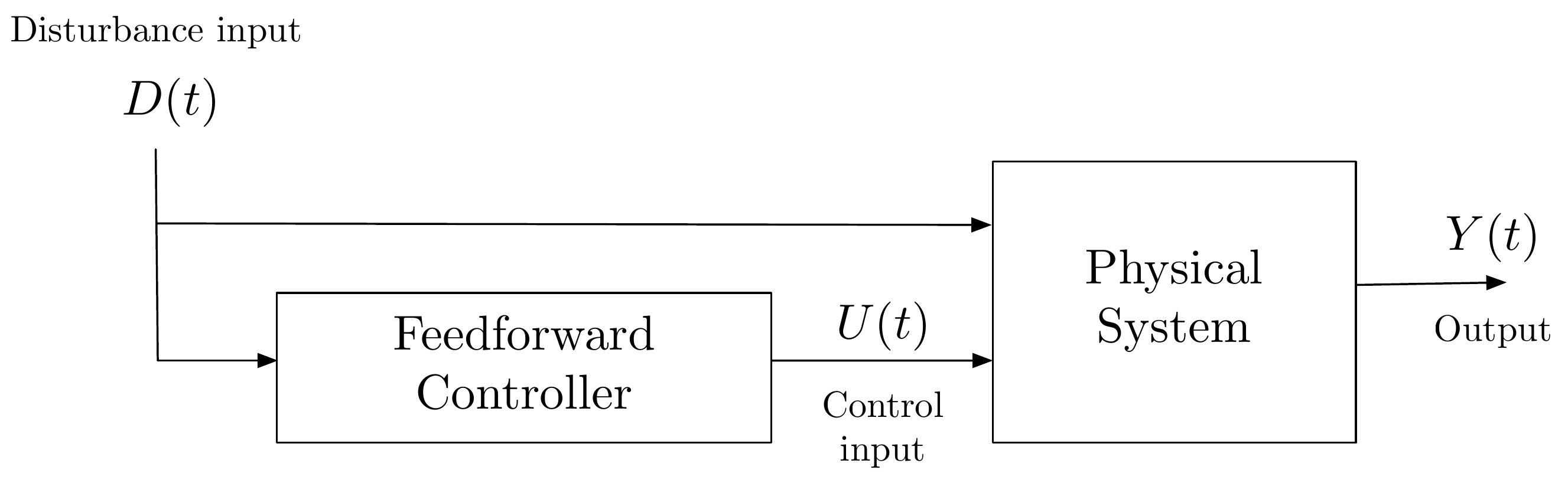}
  \centering
  \caption{Configuration of the control system with feedforward.}
  \label{configFFsystem}
\end{figure}

\section{A preliminary simple linear case} \label{sectionlinear}

Let us first examine the special case where the physical system is a simple 2$\times$2 linear hyperbolic system of the form
\begin{equation} \label{modlin}
\bmat H_t + Q_x = 0 \\[0.5em]
Q_t + \big(a H + b Q \big)_{\! x} = 0, \emat
\end{equation}
where $a$ and $b$ are two real positive constants, with boundary conditions
\begin{equation} \label{bclin}
Q(t,0) = D(t), \hh Q(t,L) - \gamma H(t,L) = U(t),	
\end{equation}
where $\gamma$ is a real constant.

The system is hyperbolic with one positive and one negative characteristic velocity which are defined as
\begin{equation}
\lambda_1 = \dfrac{b + \sqrt{b^2 + 4a}}{2} \;\; \text{and} \;\; -\lambda_2 = \dfrac{b - \sqrt{b^2 + 4a}}{2}.
\end{equation}
The reason for beginning in this way is that the simple linear system \eqref{modlin}, \eqref{bclin} allows an explicit and complete mathematical analysis of the feedforward control design. It is therefore an excellent starting point before to address the general nonlinear case for which less explicit and more complicated solutions will be discussed later on. 

\subsection{Feedforward control design in the frequency domain}

In order to solve the feedforward control problem for the linear system \eqref{modlin}, \eqref{bclin}, we introduce the {\textit {Riemann coordinates}} defined by the following change of coordinates:
\begin{equation} \label{riemcoor}
\begin{split} 
R_1(t,x) &= Q(t,x) - D(0) + \lambda_2 (H(t,x) - H^*_L), \\
R_2(t,x) &= Q(t,x) - D(0) - \lambda_1 (H(t,x) - H^*_L), 
\end{split} \end{equation}
where, as mentioned in the introduction, $H^*_L$ denotes the set point for the output variable $H(t,L)$.

This change of coordinates is inverted as follows:
\begin{equation} \label{coinv}
\begin{split} 
H(t,x) - H^*_L &=  \frac {R_1(t,x) - R_2(t,x)}{\lambda_1 + \lambda_2}, \\[0.5em]
Q(t,x) - D(0) &= \frac {\lambda_1 R_1(t,x) + \lambda_2 R_2(t,x)}{\lambda_1 + \lambda_2}.
\end{split}
\end{equation}
With these Riemann coordinates, the system \eqref{modlin} is rewritten in characteristic form as a set of two transport equations:
\begin{equation} \label{dimod}
\begin{split} 
&\p_t R_1(t,x) + \lambda_1 \p_x R_1(t,x) = 0, \\
&\p_t R_2(t,x) - \lambda_2 \p_x R_2(t,x) = 0, 
\end{split}
\end{equation}
or, equivalently, as the following two delay equations:
\begin{equation} \label{eqdelay} 
\begin{split} 
&R_1(t,L) = R_1(t-\tau_1,0) \;\text{ with } \;\tau_1 = L/\lambda_1, \\
&R_2(t,0) = R_2(t- \tau_2,L) \;\text{ with } \;\tau_2 = L/\lambda_2.
\end{split}
\end{equation}
Taking the Laplace transform of \eqref{eqdelay}, the system is written as follows in the frequency domain (with ``s'' the Laplace complex variable):
\begin{equation} \label{sdelay}
\begin{split} 
&R_1(s,L) = e^{-s \tau_1} R_1(s,0), \\
&R_2(s,0) = e^{-s \tau_2} R_2(s,L).
\end{split}
\end{equation}
Moreover, the inputs $D(t)$ and $U(t)$ are represented in the frequency domain by the following Laplace transforms:
\begin{equation} \label{lapinputs}
\widetilde U(s) = \mathcal{L}\big( U(t) - D(0) + \gamma H^*_L \big), \hd \widetilde D(s)= \mathcal{L} \big( D(t) - D(0)\big),	
\end{equation}
where $\mathcal{L}$ denotes the Laplace transform operator. 

Then, using the boundary conditions \eqref{bclin} with the change of coordinates \eqref{coinv}, the system equations \eqref{sdelay} and the definitions \eqref{lapinputs}, the input-output system dynamics are computed as follows~:
\begin{equation} 
Y(s) = P_{\! o}(s) \widetilde U(s) + P_{\! d}(s)\widetilde D(s)  \label{io2},
\end{equation}
with, for $\tau = \tau_1 + \tau_2$, the two transfer functions:
\begin{align}
&P_{\! o}(s) = - \dfrac{\lambda_1 + \lambda_2 e^{-s \tau}} {\lambda_1(\gamma + \lambda_2) + \lambda_2 (\gamma - \lambda_1) e^{-s \tau}} \;, \label{Po(s)}\\[1em]
&P_{\! d}(s) = \dfrac{(\lambda_1 + \lambda_2) e^{-s \tau_1}} {\lambda_1(\gamma + \lambda_2 ) + \lambda_2 (\gamma - \lambda_1) e^{-s \tau}} \;.	\label{Pd(s)}
\end{align}
Let us now assume uniform initial conditions at time $t=0$:
\begin{equation} \label{initcond}
H(0,x) = H^*_L, \hd Q(0,x) = D(0), \hd \text{for all } x \in [0,L].	
\end{equation}
Then it follows from \eqref{io2} that, in order to satisfy the control objective $H(t,L) = H^*_L$ $\forall t > 0$, or equivalently in Laplace coordinates $Y(s) = 0$ $\forall s $, we must select the feedforward boundary control law $\widetilde U(s)$ in function of the boundary disturbance $\widetilde D(s)$ such that
\begin{equation} \label{controlfreq}
\widetilde U(s) =  P_{\! c}(s) \, \widetilde D(s) \hd \text{with} \hd  P_{\! c}(s) = - P^{-1}_{\! o}(s) \, P_{\! d}(s) = \dfrac{(\lambda_1 + \lambda_2) e^{-s\tau_1}}{\lambda_1 + \lambda_2 e^{-s\tau}}.
\end{equation}
or, in the time domain, as:
\begin{equation} \label{contlawtemp}
U(t) = - \dfrac{\lambda_2}{\lambda_1} U(t - \tau) +  \left(1 + \dfrac{\lambda_2}{\lambda_1}\right) D(t - \tau_1) - \gamma \left(1 + \dfrac{\lambda_2}{\lambda_1}\right)H^*_L.
\end{equation}
Hence, we can see that the feedforward controller $P_{\! c}(s)$ seems to achieve the desired purpose. There is however an important limitation: the result is obtained under the assumption that the initial condition \eqref{initcond} is uniform and that the initial output density $H(0,L)$ is already at the set point $H^*_L$. If this assumption is not verified, then initial transients may appear. Obviously, such transients will vanish exponentially only if the system is exponentially stable. Hence, the practical implementation of the feedforward controller clearly requires the stability of the transfer functions of the system. From \eqref{Po(s)}, \eqref{Pd(s)} we can see that the transfer functions  $P_{\! o}(s)$ and $P_{\! d}(s)$ have the same poles that are stable if and only if $\lambda_1$, $\lambda_2$ and $\gamma$ satisfy the following inequality 
\begin{equation} \label{condition1}
 \left|\dfrac{\gamma - \lambda_1}{\gamma + \lambda_2}\right| < \dfrac{\lambda_1}{\lambda_2}.	
\end{equation}
Moreover the transfer function $P_{\! c}(s)$ of the controller has stable poles if and only if
\begin{equation} \label{condition2}
	\dfrac{\lambda_2}{\lambda_1} < 1.
\end{equation}  
These conditions imply in particular that, starting from any arbitrary initial condition,  the states of the physical system and the controller are bounded  and that $H(t,L)$ asymptotically converges to the set-point
\begin{equation}
\lim_{t \rightarrow \infty} H(t,L) = H^*_L,	
\end{equation}
such that the feedforward control objective is achieved as soon as the initial transients have vanished. Remark that in the case where Condition \eqref{condition1} is not satisfied, the feedforward controller can be combined with a feedback controller in order to ensure the global system stability. However, for the sake of simplicity and clarity, we will not study this issue in more detail in this article.

\subsection{Feedforward control design in the time domain}

Let us now show that the same feedforward control law can also be computed in another way, as the output of the following copy of the system \eqref{mod1}-\eqref{mod2}:
\begin{equation} \label{modcontrol}
\bmat \widehat H_t + \widehat Q_x = 0,  \\[0.5em]
\widehat Q_t +  \big(a \widehat H + b \widehat Q\big)_{\! x} = 0, \emat
\end{equation}
with the boundary conditions:
\begin{equation} \label{bcc}
\begin{split}
&\widehat Q(t,0) = D(t), \\ &\widehat H(t,L) = H^*_L.	
\end{split}
\end{equation}
We are going to show that the control $U(t)$ given by \eqref{contlawtemp} is equivalently given by:
\begin{equation} \label{bccontrol}
U(t) = \widehat Q(t,L) - \gamma H^*_L.	
\end{equation}
Again we use the Riemann coordinates 
\begin{equation} \label{riemcor2}
\begin{split}
\widehat R_1(t,x) &= \widehat Q(t,x) - D(0) + \lambda_2 (\widehat H(t,x) - H^*_L), \\[0.5em]
\widehat R_2(t,x) &= \widehat Q(t,x) - D(0) - \lambda_1 (\widehat H(t,x) - H^*_L),
\end{split}
\end{equation}
which are inverted as
\begin{equation}\label{invriemcor2}
\begin{split}
&\widehat H(t,x) - H^*_L = \dfrac{\widehat R_1(t,x) - \widehat R_2(t,x)}{\lambda_1 + \lambda_2},  \\[0.5em]
&\widehat Q(t,x) - D(0) = \dfrac{\lambda_1 \widehat R_1(t,x) + \lambda_2 \widehat R_2(t,x)}{\lambda_1 + \lambda_2}.
\end{split}
\end{equation}
In these coordinates, the system \eqref{modcontrol} is written as follows:
\begin{equation}
\begin{split}
&\p_t \widehat R_1(t,x) + \lambda_1 \p_x \widehat R_1(t,x) = 0, \\
&\p_t \widehat R_2(t,x) - \lambda_2 \p_x \widehat R_2(t,x) = 0,
\end{split}
\end{equation}
or equivalently as the delay system:
\begin{subequations} 
\begin{align}
&\widehat R_1(t,L) = \widehat R_1(t-\tau_1, 0), \label{delay1}\\
&\widehat R_2(t,0) = \widehat R_2(t-\tau_2, L).	\label{delay2}
\end{align}
\end{subequations}
In the Riemann coordinates, the boundary conditions \eqref{bcc} are written:
\begin{subequations} \label{bcriemcont}
\begin{align} 
&\widehat R_1(t,0) = - \dfrac{\lambda_2}{\lambda_1} \widehat R_2(t,0) + \left(1 + \dfrac{\lambda_2}{\lambda_1} \right) (D(t) - D(0)), \label{bcriem1}\\
&\widehat R_2(t,L) = \widehat R_1(t,L).	 \label{bcriem2}
\end{align}
\end{subequations}
From \eqref{invriemcor2}, \eqref{delay1} and \eqref{bcriem2}, we have:
\begin{align}  
\widehat Q(t,L) = \dfrac{\lambda_1 \widehat R_1(t,L) + \lambda_2 \widehat R_2(t,L)}{\lambda_1 + \lambda_2} + D(0) &= \widehat R_1(t,L) + D(0) \nonumber \\[0.5em] &= \widehat R_1(t-\tau_1,0) + D(0). \label{eqaux1}
\end{align}
Since the boundary condition \eqref{bcriem1} is obviously valid at any time instant, and using 
\eqref{eqaux1}, we have:
\begin{align} 
\widehat Q(t,L) &= \widehat R_1(t - \tau_1, 0) + D(0) \nonumber \\[0.5em] &= - \dfrac{\lambda_2}{\lambda_1} \widehat R_2(t-\tau_1,0) + \left(1 + \dfrac{\lambda_2}{\lambda_1} \right) \Big(D(t-\tau_1) - D(0) \Big).	\label{eqaux2}
\end{align}
Moreover, using successively \eqref{delay2}, \eqref{bcriem2} and \eqref{eqaux1}, we have:
\begin{align} 
\widehat R_2(t-\tau_1, 0) = \widehat R_2(t - \tau, L) &= \widehat R_1(t-\tau, L) = \widehat Q(t-\tau, L) - D(0). \label{eqaux3}	
\end{align}
Then, combining \eqref{eqaux2} and \eqref{eqaux3}, we get:
\begin{equation}
\widehat Q(t,L) = - \dfrac{\lambda_2}{\lambda_1} \widehat Q(t-\tau,L) + \left(1 + \dfrac{\lambda_2}{\lambda_1} \right) D(t-\tau_1).	
\end{equation}
Comparing this equation with \eqref{contlawtemp}, we finally conclude that, as announced, the control law is given by:
\begin{equation} \label{intercon}
U(t) = \widehat Q(t,L) - \gamma H^*_L.
\end{equation}
This expression of the feedforward control law in the time domain is of special interest to motivate its extension to nonlinear systems as we shall see in the next section.

%\newpage
\section{The general nonlinear case}

In this section, we now address the feedforward control design problem for the general nonlinear physical system of Section 1, i.e.
\begin{equation} \label{physyst}
\begin{matrix}   
H_t + Q_x = 0, \\[0.5em]
Q_t + (f(H,Q))_x + g(H,Q) = 0, 
\end{matrix}
\end{equation}
\vspace{0.2em}
\begin{equation}\label{physystBC}
\begin{matrix} 
\alpha(H(t,0), Q(t,0)) = D(t),  \\[0.5em]
\beta( H(t,L), Q(t,L)) = U(t).  
\end{matrix}
\end{equation}
The objective is to keep the output density $H(t,L)$ at the setoint $H^*_L$ despite the disturbance $D(t)$. 

We assume that all the required conditions are met for this system to be well posed and have a unique solution in the domain of interest. The existence and uniqueness of solutions is a topic which is the subject of numerous publications. We do not address this issue in this article but we refer the reader to the paper \cite{Wan06} by Zhiqiang Wang (and the references therein) where explicit conditions for hyperbolic systems of the form \eqref{physyst}, \eqref{physystBC} are given.

\subsection{Feedforward control design}
Obviously the frequency method is not relevant for a nonlinear system such as \eqref{physyst}, \eqref{physystBC}. However, from our analysis of the linear case in the previous section, a natural way  to generalize the  control design to nonlinear systems in the time domain is as follows. We use a copy of the system \eqref{physyst}:
\begin{equation} \label{feedcont}
\begin{matrix} 
\widehat H_t + \widehat Q_x = 0,  \\[0.5em]
\widehat Q_t + (f(\widehat H,\widehat Q))_x + g(\widehat H,\widehat Q) = 0, 
\end{matrix}
\end{equation}
with the boundary conditions:
\begin{equation} \label{bcfeedcont}
\begin{matrix} 
\alpha(\widehat H(t,0), \widehat Q(t,0)) = D(t), \\[0.5em]
\widehat H(t,L) = H^*_L,	
\end{matrix}
\end{equation}
and with the feedforward control defined as:
\begin{equation} \label{nlcontlaw}
	U(t) = \beta(H^*_L, \widehat Q(t,L)).
\end{equation}
In the next theorem, it is shown that this feedforward controller  \eqref{feedcont}, \eqref{bcfeedcont}, \eqref{nlcontlaw} achieves the desired purpose.
\begin{theorem} \label{theorem1}
Assume that both systems \eqref{physyst}, \eqref{physystBC} and \eqref{feedcont}, \eqref{bcfeedcont} are interconnected with the control law \eqref{nlcontlaw} and have the same initial condition
\begin{equation} \label{initnl}
H(0,x) = \widehat H(0,x), \;\; Q(0,x) = \widehat Q(0,x), \;\; \text{for all } x \in [0,L],
\end{equation}
with $H(0,L) = \widehat H(0,L) = H^*_L$.
Then, for all positive $t$ it holds that $H(t,L) = H^*_L$.
\end{theorem}
\begin{proof}
Let us first observe that the condition \eqref{initnl} implies not only that the two systems have the same initial condition but also that they have identical boundary conditions at the initial time $t=0$:
\begin{align}
&D(0) = \alpha(H(0,0), Q(0,0)) = \alpha(\widehat H(0,0), \widehat Q(0,0)) , \\[0.5em]
&U(0) = \beta(H^*_L, Q(0,L)) = \beta(H^*_L, \widehat Q(0,L)).	
\end{align}
Then it is immediately clear that the solution $\widehat H(t,x), \widehat Q(t,x)$ (with $\widehat H(t,L) = H^*_L \, \forall t$) of the system \eqref{feedcont}, \eqref{bcfeedcont} is also a possible solution of the system \eqref{physyst}, \eqref{physystBC}, i.e $H(t,x) = \widehat H(t,x)$ and $Q(t,x)) = \widehat Q(t,x)$ for all $t$ and $x$. Since the solution of the system \eqref{physyst}, \eqref{physystBC} is unique, the result follows.
\end{proof}

This theorem shows however that the stability issue mentioned in the linear case is still present here. The feedforward control leads to an exact cancellation of the disturbance for all $t \gs 0$ only if the physical system and the feedforward controller have exactly identical initial conditions. Otherwise some sort of stability of the solutions is required to guarantee an asymptotic decay of the initial transients. As already pointed out by T. Glad \cite{Gla19} for finite dimensional nonlinear systems, Lyapunov theory helps to discuss this stability issue as we shall see in the next section.  

\subsection{Stability conditions}
Our purpose in this section is to derive sufficient stability conditions for the overall control system  \eqref{physyst}, \eqref{physystBC}, \eqref{feedcont}, \eqref{bcfeedcont}, \eqref{nlcontlaw}.

A steady state is a system solution that does not change over time. We assume that, for any set point $H^*_L$ and any given constant disturbance input $D(t) = D^*$ for all $t$, the system has a unique well-defined steady state $H(t,x) = \widehat H(t,x) = H^*(x)$, $Q(t,x) = \widehat Q(t,x) = Q^*$ for all $t$. The steady state flux density $Q^*$ is uniform on the domain $[0,L]$ and the steady state density function $H^*(x)$ is a solution of the ordinary differential equation
\begin{equation} \label{ss}
	\big(f(H^*,Q^*)\big)_{\! x} + g(H^*,Q^*) = 0, \hd H^*(L) = H^*_L, \hd x \in[0,L].
\end{equation}
In order to linearize the system, we define the deviations of the disturbance input $D(t)$ and  the states $H(t,x)$, $\widehat H(t,x)$, $Q(t,x)$, $\widehat Q(t,x)$ with respect to the steady states $D^*$, $H^*(x)$ and $Q^*$:
\begin{equation}
\begin{split}
&d(t) = D(t) - D^*, \\[0.5em]
	&h(t,x) = H(t,x) - H^*(x), \hd q(t,x) = Q(t,x) - Q^*, \\[0.5em]
	&\hat h(t,x) = \widehat H(t,x) - H^*(x), \hd \hat q(t,x) = \widehat Q(t,x) - Q^*.
	\end{split}
\end{equation}
With these notations the linearization of the physical system \eqref{physyst}, \eqref{physystBC}
 about the steady state is 
\begin{equation} \label{systdev}
\begin{split}
&h_t + q_x = 0, \\[0.5em]
&q_t + a(x) h_x + b(x) q_x + \big(a_x(x) + \tilde a(x)\big)h + \big(b_x(x) + \tilde b(x)\big)q = 0, 
\end{split} 
\end{equation}
with the boundary conditions
\begin{equation} \label{bcsystdev}
\begin{split}	
	& \alpha_{h} h(t,0) + \alpha_{q} q(t,0) = d(t), \\[0.5em]
	& \beta_{h} h(t,L) + \beta_{q} q(t,L) = \beta_{q}  \hat q(t,L). 
\end{split}
\end{equation}
In these equations, we use the following notations:
\begin{equation}
\begin{split}	
	&a(x) = \dfrac{\p f}{\p H}(H^*(x), Q^*), \hd b(x) = \dfrac{\p f}{\p Q}(H^*(x), Q^*), \\[0.5em]
	&\tilde a(x) = \dfrac{\p g}{\p H}(H^*(x), Q^*), \hd \tilde b(x) = \dfrac{\p g}{\p Q}(H^*(x), Q^*), \\[0.5em]
	&\alpha_{h} = \dfrac{\p \alpha}{\p H}(H^*(0), Q^*), \hd \alpha_{q} = \dfrac{\p \alpha}{\p Q}(H^*(0), Q^*) \\[0.5em]
	&\beta_{h} = \dfrac{\p \beta}{\p H}(H^*_L, Q^*), \hd  \beta_{q} = \dfrac{\p \beta}{\p Q}(H^*_L, Q^*).
\end{split}
\end{equation}
Similarly, the linearization of the controller \eqref{feedcont}, \eqref{bcfeedcont}, \eqref{nlcontlaw}
 about the steady state is 
 \begin{equation} \label{contdev}
\begin{split}
&\hat h_t + \hat q_x = 0, \\[0.5em]
&\hat q_t + a(x) \hat h_x + b(x) \hat q_x + \big(a_x(x) + \tilde a(x)\big)\hat h + \big(b_x(x) + \tilde b(x)\big)\hat q = 0,
\end{split} 
\end{equation}
with the boundary conditions
\begin{equation} \label{bccontdev}
\begin{split}
	&\alpha_{h} \hat h(t,0) + \alpha_{q} \hat q(t,0) = d(t), \\[0.5em]
	&\hat h(t,L) = 0.
\end{split}
\end{equation}
Let us now introduce the following notations for the deviations between the states of the physical system and the controller:
\begin{equation} \begin{split}
&\tilde h(t,x) = H(t,x) - \widehat H(t,x) = h(t,x) - \hat h(t,x), \\[0.5em] &\tilde q(t,x) = Q(t,x) - \widehat Q(t,x) = q(t,x) - \hat q(t,x).	
\end{split} \end{equation}
Then, from \eqref{systdev}, \eqref{bcsystdev} we have the following linear `error' system
\begin{equation} \begin{split} \label{perturblindev}
&\tilde h_t + \tilde q_x = 0, \\[0.5em]
&\tilde q_t + a(x) \tilde h_x + b(x) \tilde q_x + \big(a_x(x) + \tilde a(x)\big)\tilde h + \big(b_x(x) + \tilde b(x)\big)\tilde q = 0, 
\end{split} \end{equation}
with the boundary conditions
\begin{equation} \label{bcperturblindev}\begin{split}
	& \alpha_{h} \tilde h(t,0) + \alpha_{q} \tilde q(t,0) = 0, \\[0.5em]
	& \beta_{h} \tilde h(t,L) + \beta_{q} \tilde q(t,L) = 0.
\end{split}\end{equation}
This error system has clearly a unique uniform steady-state $\tilde h(t,x) \equiv 0$, $\tilde q(t,x) \equiv 0$.
With the definitions
\begin{equation} \begin{split}
&\tilde z = \bpm \tilde h \\[0.5em] \tilde q \epm, \;\; A(x) = \bpm 0 & 1 \\[0.5em] a(x) & b(x) \epm, \\[0.5em]  &B(x) = \bpm 0 & 0 \\[0.5em] a_x(x) + \tilde a(x) & b_x(x) + \tilde b(x) \epm, \end{split}
\end{equation}
the system \eqref{perturblindev} is rewritten in matrix form:
\begin{gather} \label{perturblindevmat}
\tilde z_t + A(x) \tilde z_x +	B(x) \tilde z = 0.
\end{gather}
Since the system is supposed to be hyperbolic,
it is assumed that the matrix $A(x)$ has two real distinct eigenvalues
\begin{equation}
\lambda_1(x) = \dfrac{b(x) + \sqrt{b^2(x) + 4a(x)}}{2} \;\; \text{and} \;\; -\lambda_2(x) = \dfrac{b(x) - \sqrt{b^2(x) + 4a(x)}}{2},
\end{equation}
with
\begin{equation}
	b^2(x) + 4a(x) > 0 \;\; \text{for all } x \in [0,L].
\end{equation}
Remark that
\begin{gather}
\lambda_1(x) > 0, \;\; \lambda_2(x) > 0, \\[0.5em]
a(x) = \lambda_1(x)\lambda_2(x), \;\; b(x) = \lambda_1(x) - \lambda_2(x), \\[0.5em] b^2(x) + 4a(x) = \big(\lambda_1(x) + \lambda_2(x)\big)^2.
\end{gather}
Therefore, for all $x \in [0,L]$, the matrix $A(x)$ can be diagonalized with the invertible matrix $N(x)$ defined as
\begin{equation}
	N(x) = \bpm \lambda_2(x) & 1 \\[0.5em] - \lambda_1(x) & 1 \epm
\end{equation}
such that
\begin{equation} \label{N(x)}
	 N(x)A(x) = \Lambda(x) N(x) \hd \text{with} \hd \Lambda(x) = \bpm \lambda_1(x) & 0 \\[0.5em] 0 & -\lambda_2(x) \epm.
\end{equation}
In order to address the stability of the linear system \eqref{perturblindev}, \eqref{bcperturblindev}, we introduce the following basic quadratic Lyapunov function candidate:
\begin{equation} \label{basicLyap}
\mb{V} = \int_0^L \big(\tilde z^T P(x) \tilde z\big) dx
\end{equation}
with $P(x)$ a symmetric positive definite matrix of the form
\begin{equation} \label{P(x)}
P(x) = N^T(x) \Delta(x) N(x)	, \;\; \Delta(x) =  \bpm p_1(x) & 0 \\ 0 & p_2(x) \epm
\end{equation}
where $p_i : [0,L] \rightarrow \real_+$ $(i = 1,2)$ are two real positive functions to be determined.

We compute the time derivative of $\mb{V}$ along the $C^1$-solutions of the system \eqref{perturblindev}, \eqref{bcperturblindev}:
\begin{equation} \begin{split} \label{dVdt}
\dfrac{d\mb{V}}{dt} &= \int_0^L  \big(\tilde z^T P(x) \tilde z_t + \tilde z_t^T P(x) \tilde z\big) dx  \\[0.5em]
&= - \int_0^L \Big(\tilde z^T P(x) \big(A(x)\tilde z_x + B(x)\tilde z\big) + \big(\tilde z_x^T A^T(x) + \tilde z^T B^T(x)\big) P(x) \tilde z \Big) dx. 
\end{split} \end{equation}
Using \eqref{N(x)} and \eqref{P(x)}, we see that the matrix $M(x) = P(x)A(x)$ is symmetric:
\begin{equation} \label{sym}
M(x) = P(x)A(x) = A^T(x)P(x) = N^T(x) \Delta(x) \Lambda(x) N(x).
\end{equation}
Then, from \eqref{dVdt}, \eqref{sym} and using integration by parts, we have
\begin{equation} \label{dVdt2}
\begin{split}
\dfrac{d\mb{V}}{dt} &=  \tilde z^T(t,0) M(0) \tilde z(t,0) -  \tilde z^T(t,L) M(L) \tilde z(t,L) \\[0.5em]
&- 	\int_0^L \tilde z^T \Big( -M_x(x) + B^T(x)P(x) + P(x) B(x) \Big) \tilde z \, dx.
\end{split}
\end{equation}
Under the boundary conditions \eqref{bcperturblindev}, it can be checked that $\tilde z^T(t,0) M(0) \tilde z(t,0) < 0$ if
\begin{enumerate}
\item [(a1)] $\hspace{1.5cm} \left( \dfrac{\alpha_h - \alpha_q \lambda_2(0)}{\alpha_h + \alpha_q \lambda_1(0)} \right)^{\! 2} < \dfrac{p_2(0)}{p_1(0)} \dfrac{\lambda_2(0)}{\lambda_1(0)}$,
\end{enumerate} 
and that $- \tilde z^T(t,L) M(L) \tilde z(t,L) < 0$ if
\begin{enumerate}
\item [(a2)] $\hspace{1.5cm} \left( \dfrac{\beta_h + \beta_q \lambda_1(L)}{\beta_h - \beta_q \lambda_2(L)} \right)^{\! 2} < \dfrac{p_1(L)}{p_2(L)}  \dfrac{\lambda_1(L)}{\lambda_2(L)}  $.
\end{enumerate}
Hence, from \eqref{dVdt2}, it follows that if the two positive functions $p_i \in C^1([0,L], (0, +\infty))$ $(i = 1,2)$ can be selected such that conditions (a1) and (a2) are satisfied and 
\begin{enumerate}
\item[(b)] the matrix $-M_x(x) + B^T(x)P(x) + P(x) B(x)$ is positive definite for all $x \in [0,L]$,
\end{enumerate}
then $d\mb{V}/dt$ is a negative definite function along the solutions of the system \eqref{perturblindev}, \eqref{bcperturblindev}, which induces the following stability property because $\mb{V}$ is equivalent to a $L^2$ norm for $\tilde z(t,.) \in L^2([0,L], \real^2)$.
\begin{theorem} \label{theorem2}
If there exist two functions $p_i \in C^1([0,L], (0, +\infty))$ $(i = 1,2)$ such that conditions (a1), (a2) and (b) are satisfied, then the system \eqref{perturblindev}, \eqref{bcperturblindev} is $L^2$-exponentially stable, that is there exist two positive constants $C$ and $\nu$ such that, from any initial condition $\tilde z(0,.) \in L^2([0,L], \real^2)$, the system solution satisfies the inequality
\begin{equation}
	\|\tilde z(t,.) \|_{L^2} \ls Ce^{-\nu t}\|\tilde z(0,.)\|_{L^2}, \hd t \in [0,+\infty).
\end{equation}\qed
\end{theorem}
This theorem tells us that the solution of the (linearized) physical system asymptotically tracks the solution of the (linearized) controller system, independently of the disturbance. In particular the theorem implies that the output $H(t,L)$ asymptotically converges to the set-point:
\begin{equation}
\lim_{t \rightarrow \infty} H(t,L) = H^*_L	
\end{equation}
whatever the size and the shape of the disturbance.
This can be viewed as a generalization of the condition \eqref{condition1} which was obtained in the simple linear case addressed in Section \ref{sectionlinear}. However, this is not sufficient to conclude that the feedforward control objective is achieved because we are not yet guaranteed that all the initial transients of the overall system will exponentially vanish under the conditions stated in Theorem \ref{theorem2}. Indeed, we have to verify in addition that the solutions of the controller itself are not unstable.

For that purpose, let us thus consider the linearized controller system \eqref{contdev} with the boundary conditions \eqref{bccontdev}. It can be observed that the system dynamics are very similar, but not equal, to those of the error system.  The only difference lies in the boundary conditions. We can therefore use the same Lyapunov function candidate
\begin{equation} \label{contLyap}
\mb{\hat V} = \int_0^L \big(\hat z^T P(x) \hat z\big) dx \hh \Big(\text{where} \;\; \hat z := (\hat h, \hat q )^T \; \Big)
\end{equation}
for which we have
\begin{equation} \label{dVdt3}
\begin{split}
\dfrac{d\mb{\hat V}}{dt} &= \hat z^T(t,0) M(0) \hat z(t,0) -  \hat z^T(t,L) M(L) \hat z(t,L)  \\[0.5em]
&- 	\int_0^L \hat z^T \Big( -M_x(x) + B^T(x)P(x) + P(x) B(x) \Big) \hat z \, dx.
\end{split}
\end{equation}
Under the boundary condition \eqref{bccontdev}, we have
\begin{equation} \label{bc31}
	-  \hat z^T(t,L) M(L) \hat z(t,L) = -\big(p_1(L)\lambda_1(L) - p_2(L)\lambda_2(L)\big)\hat q^2(t,L)
\end{equation}
which is negative if and only if
\begin{itemize}
\item[(a3)] \hspace{1.5cm} $\dfrac{p_2(L)\lambda_2(L)}{p_1(L)\lambda_1(L)} < 1$.	
\end{itemize}
Remark that this condition can be viewed as a generalization of condition \eqref{condition2} which was obtained from a frequency domain approach for the simple linear example of Section \ref{sectionlinear}.

Let us now assume that $\alpha_q \neq 0$ (the case $\alpha_q = 0$ will be considered next). Then, under the boundary condition \eqref{bccontdev}, we have
\begin{equation} \label{bc32}
	\hat z^T(t,0) M(0) \hat z(t,0) = - \gamma_0 \hat h^2(t,0) + \gamma_1 \hat h(t,0) d(t) + \gamma_2 d^2(t)
\end{equation}
with
\begin{gather}
\gamma_0 = - p_1(0) \lambda_1(0)\left(\lambda_2(0) - \dfrac{\alpha_h}{\alpha_q}\right)^{\!\! 2}	 + p_2(0) \lambda_2(0) \left(\lambda_1(0) + \dfrac{\alpha_h}{\alpha_q}\right)^{\!\! 2}, \\[0.5em]
\gamma_1 = \dfrac{2}{\alpha_q} \Big[ p_1(0) \lambda_1(0)\left(\lambda_2(0) - \dfrac{\alpha_h}{\alpha_q}\right) + p_2(0) \lambda_2(0) \left(\lambda_1(0) + \dfrac{\alpha_h}{\alpha_q}\right) \Big], \\[0.5em]
\gamma_2 = \dfrac{1}{\alpha^2_q}\big(p_1(0)\lambda_1(0) - p_2(0)\lambda_2(0)\big).
\end{gather}
In the case where $\alpha_q = 0$ and necessarily $\alpha_h \neq 0$\footnote{The case where both $\alpha_h$ and $\alpha_q$ would equal zero is pointless in the context of this paper because it would correspond to a system without disturbance.}, we have
\begin{equation}
	\hat z^T(t,0) M(0) \hat z(t,0) = - \gamma_0 \hat q^2(t,0) + \gamma_1 \hat q(t,0) d(t) + \gamma_2 d^2(t)
\end{equation}
with
\begin{gather}
\gamma_0 = - p_1(0) \lambda_1(0) + p_2(0) \lambda_2(0), \\[0.5em]
\gamma_1 = \dfrac{2 \lambda_1(0) \lambda_2(0)}{\alpha_h} \big( p_1(0)  + p_2(0)  \big), \\[0.5em]
\gamma_2 = \dfrac{\lambda_1(0) \lambda_2(0)}{\alpha^2_h}\big(p_1(0)\lambda_2(0) - p_2(0)\lambda_1(0)\big).
\end{gather}
In both cases, it can be verified that $\gamma_0 > 0$ if and only if condition (a1) is verified. 
From, \eqref{dVdt3}, \eqref{bc31}, \eqref{bc32}, if conditions (a1), (a3) and (b) are satisfied, we can write
\begin{equation}
	\dfrac{d\mb{\hat V}}{dt} \ls - \mu_0 \int_0^L \big(\hat z^T \hat z\big) dx - \gamma_0 \hat q^2(t,0) + |\gamma_1| |\hat q(t,0)| |d(t)| + |\gamma_2| d^2(t).
\end{equation}
where $\mu_0 > 0$ is the infimum over $[0,L]$ of the eigenvalues of the positive definite matrix $-M_x(x) + B^T(x)P(x) + P(x) B(x)$.

Let us now remark that
\begin{equation}
	|\gamma_1||\hat q(t,0)| |d(t)| \ls \dfrac{\gamma_0}{2}\hat q^2(t,0) + \dfrac{\gamma_1^2}{2\gamma_0} d^2(t). 
\end{equation}
Therefore:
\begin{align}
	\dfrac{d\mb{\hat V}}{dt} &\ls - \mu_0 \int_0^L \big(\hat z^T \hat z\big) dx - \dfrac{\gamma_0}{2} \hat q^2(t,0) + \left(\dfrac{\gamma_1^2}{2\gamma_0} + |\gamma_2|\right) d^2(t) \\[0.5em]
	&\ls - \dfrac{\mu_0}{\mu_1} \mb{\hat V} + \left(\dfrac{\gamma_1^2}{2\gamma_0} + |\gamma_2|\right) d^2(t) \label{auxi}
\end{align}
where $\mu_1$ is the infimum over $[0,L]$ of the eigenvalues of the positive definite matrix $P(x)$.

Since $\mb{\hat V}$ is equivalent to the square of the $L^2$ norm for $\hat z(t,.) \in L^2([0,L], \real^2)$, if the disturbance input is bounded, this induces the input-to-state stability property stated in the following proposition.
\begin{theorem} \label{theorem3}
	If there exist two functions $p_i \in C^1([0,L], (0, +\infty))$ $(i = 1,2)$ such that conditions (a1), (a3) and (b) are satisfied, then the system \eqref{perturblindev}, \eqref{bcperturblindev} is $L^2$-input-to-state stable, that is there exist three positive constants $C_1$, $C_2$ and $\nu$ such that, from any initial condition $\hat z(0,.) \in L^2([0,L], \real^2)$, the system solution satisfies the inequality
\begin{equation}
\lVert \hat z(t,\cdot)\rVert_{L^{2}} \ls C_{1}\lVert \hat z(0,.) \rVert_{L^{2}} e^{-\nu t}+ C_{2} \; \sup _{t \gs 0}|d(t)|, \hd t \in [0,+\infty).
\end{equation}
\qed
\end{theorem}
Hence, if the input disturbance $d(t)$ is bounded, we can conclude from Theorems \ref{theorem2} and \ref{theorem3} that, starting from any arbitrary initial condition in $L^2$,  the states of the (linearized) physical system and the (linearized) controller are bounded in $L^2$ and that $H(t,L)$ asymptotically converges to the set-point
\begin{equation}
\lim_{t \rightarrow \infty} H(t,L) = H^*_L,	
\end{equation}
such that the feedforward control objective is achieved as soon as the initial transients have vanished.

As it is justified in many recent publications (see e.g. \cite[Theorem 6.6]{BasCor14} and \cite[Section 2.1]{BasCorHay20}), it is also worth noting that the conditions (a1), (a2), (a3) and (b) for the $L^2$-stability of the linearized system, may also be sufficient to establish the $H^2$-stability of the overall \textit{nonlinear} system in a neighbourhood of the steady-state. A rigorous detailed analysis of this generalization is however delicate and would go far beyond the scope of this article. We will limit ourselves here to a more pragmatic approach which consists in checking the applicability and the effectiveness of the method in the realistic nonlinear applications considered in the rest of the paper.

 To conclude this section, let us also mention that sufficient ISS conditions can also be established for the $C^1$-norm but the analysis is still more intricate. The interested reader is referred, among others, to the recent publications \cite[Section 9.4]{KarKrs19}, \cite{FerPri20}, \cite{YorBan20} and \cite{BasCorHay20}. However, it must be said that in many $2 \times 2$ physical control systems of practical interest, the ISS conditions are equivalent for the $H^2$ and $C^1$ norms. This will appear for instance in the next section where we apply our theory to the example of level control in an open channel.

\section{Application to level control in an open channel} \label{sectionchannel}

In the field of hydraulics, the flow in open channels is generally represented by the Saint-Venant equations which are a typical example of a $2\times2$ nonlinear hyperbolic system.

We consider the special case of a pool of an open channel  as represented in Figure \ref{bief}. We assume that the channel is horizontal
and prismatic with a constant rectangular section and a unit width.
\begin{figure}[h]
\begin{center}
\includegraphics[width=12cm]{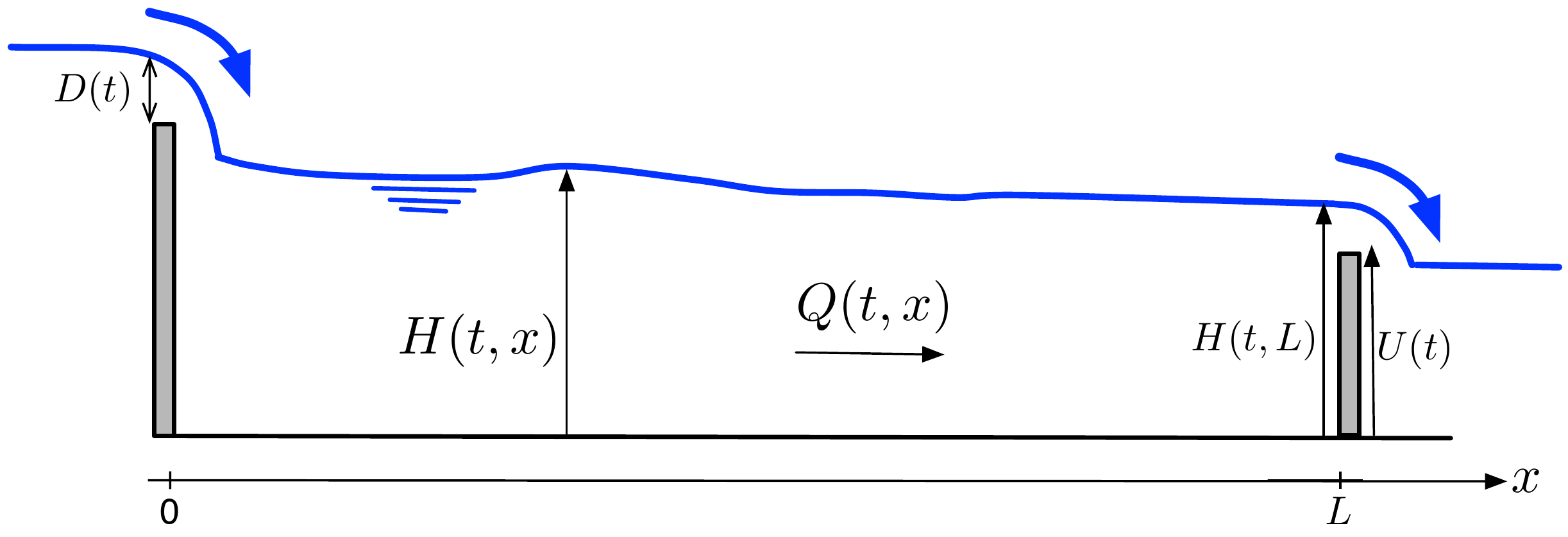}
\caption{A pool of an open channel with overshot gates at the upstream and downstream sides.}
\label{bief}
\end{center}
\end{figure}

The flow dynamics are described by the Saint-Venant equations
\begin{equation} \label{sv}
\begin{matrix}
H_t + Q_x = 0,  \\[0.5em]
Q_t + \left(\dfrac {Q^2}{H} + \text{g} \dfrac{H^2}{2}\right)_{\! \displaystyle x}+ c_f \dfrac{Q^2}{H^2}= 0. 
\end{matrix}
\end{equation}
where $H(t,x)$ represents the water level and
$Q(t,x)$ the water flow rate  in the pool while g denotes the gravitation constant and $c_f$ is an adimensional friction coefficient. This system is in the form \eqref{mod1}, \eqref{mod2} with 
\begin{equation}
	f(H,Q) = \dfrac{Q^2}{H} + \text{g} \dfrac{H^2}{2} \; \text{ and } \; g(H,Q) = c_f \dfrac{Q^2}{H^2}.
\end{equation}
The system is subject to the following boundary conditions:
\begin{equation} \label{bcsv}
\begin{split}
	&Q(t,0) = c_g\sqrt{\big[D(t)\big]^3}, \\[0.5em]
	&Q(t,L) = c_g\sqrt{\big[H(t,L)-U(t)\big]^3}.
\end{split}
\end{equation}
These boundary conditions are given by standard hydraulic models of overshot gates (see Fig.\ref{bief}). The first boundary condition imposes the value of the canal inflow rate $Q(t,0)$ as a function of the water head above the gate $D(t)$ which is the measurable input disturbance. The second boundary condition corresponds to the control overshot gate at the downstream side of the canal. The control action is the vertical elevation $U(t)$ of the gate. In both models, $c_g$ is a constant discharge coefficient. Let us also remark that these boundary conditions are written in the form \er{physystBC} as follows:
\begin{equation} \label{bcab}
\begin{matrix} 
\alpha(H(t,0), Q(t,0)) = \big(c_g^{-1} Q(t,0) \big)^{\! 2/3} = D(t),  \\[0.5em]
\beta( H(t,L), Q(t,L)) = H(t,L) - \big(c_g^{-1} Q(t,L) \big)^{\! 2/3} = U(t).  
\end{matrix}
\end{equation}

For a constant gate position $U(t) = U^* > 0$ $\forall t$ and a constant inflow rate $Q(t,0) = Q^*>0$ $\forall t$, a steady state is a time-invariant solution $H^*(x),  Q^*$ given by:
\begin{subequations} \label{svss}
\begin{align}
&H(t,x) = H^*(x) \; \text{ and } \; Q(t,x) = Q^* \hd \forall t, \;\; x\in[0,L], \\[0.5em] 
&H^*(L) = U^* + (c_g^{-1} Q^* )^{2/3}, \\[0.5em]
&H^*(x) \text{ solution of } (\text{g}H^{*3}(x) - Q^{*2})H^*_x(x) + c_f Q^{*2} = 0\label{svssc}
\end{align}
\end{subequations}
The existence of a solution to \eqref{svssc} requires that $\text{g}H^{*3}(L) \neq Q^{*2}$. If $\text{g}H^{*3}(L) > Q^{*2}$, then \eqref{svssc} has a solution (note that $H^*(x)$ is then a decreasing function of $x$ over $[0,L]$) and the steady state flow is subcritical (or fluvial). In such case, according to the physical evidence, $H^*(x)$ is positive :
\begin{equation}
H^*(x) > 0 \;\; \text{for all } x \in [0,L],
\end{equation}
and satisfies the following inequality:
\begin{equation} \label{subcritical}
0 < \text{g}H^{*3}(x) - Q^{*2}, \;\;\; \forall x \in [0,L].
\end{equation}
In the case where $\text{g}H^{*3}(L) < Q^{*2}$, the steady state, if it exists, is said to be supercritical (or torrential). We do not consider that case in the present article.

The control objective is to regulate the level $H(t,L)$  at the set point $H^*_L$, by acting on the gate position $U(t)$. More precisely, it is requested to adjust the control $U(t)$ in order to have $H(t,L) = H^*_L$ $\forall t$ in spite of the variations of the disturbing inflow rate measured by the signal
$D(t)$.

To solve this control problem, the design of linear feedforward controllers based on simplified linear models of open channels with uniform steady states, as we have introduced in Section \ref{sectionlinear}, was addressed previously in \cite{BasCordAn05},  \cite[Section 9]{LitFro09c} and \cite{LitFroSco07}. In this article, we extend these results to the general case of open channels with non linear Saint-Venant dynamics and non uniform steady-states. 

On the basis of our previous discussions, the feedforward control law is defined as follows:
\begin{equation} \label{nlcontlaw2}
U(t) = H^*_L - (c_g^{-1} \widehat Q(t,L))^{2/3}	
\end{equation}
where $\widehat Q(t,L)$ is computed with the auxiliary system dynamics as in \eqref{feedcont}, \eqref{bcfeedcont}:
\begin{equation} \label{feedcont2}
\begin{matrix} 
\widehat H_t + \widehat Q_x = 0,  \\[0.5em]
\widehat Q_t + \left(\dfrac{\widehat Q^2}{\widehat H} + \text{g} \dfrac{\widehat H^2}{2}\right)_{\!\! \displaystyle x} + c_f \dfrac{\widehat Q^2}{\widehat H^2} = 0, 
\end{matrix}
\end{equation}
\begin{equation} \label{bcfeedcont2}
\begin{matrix} 
\widehat Q(t,0) = c_g\sqrt{\big[D(t)\big]^3}, \\[0.5em]
\widehat H(t,L) = H^*_L,	
\end{matrix}
\end{equation}
To simplify the notations, we define the steady state water velocity 
\begin{equation} 
	V^*(x) = \dfrac{Q^*}{H^*(x)} > 0 \;\; \forall x \in [0,L].
\end{equation}
With this notation, the subcritical condition \eqref{subcritical} is written:
\begin{equation} \label{subcritical2}
	\text{g}H^*(x) - V^{*2}(x) > 0 \;\;\; \forall x \in[0, L].
\end{equation}
Now, from the linearization of the control system \eqref{sv}, \eqref{bcsv},  \eqref{nlcontlaw2}, \eqref{feedcont2}, \eqref{bcfeedcont2}, we have, in this application, the following matrices $A(x)$ and $B(x)$: \begin{equation}
A(x) = \bpm 	0 & 1 \\[0.5em] \text{g}H^*(x) - V^{*2}(x) & 2 V^*(x) \epm,
\end{equation}
\begin{equation}
B(x) = \bpm 	0 & & 0 \\[0.5em] -3 \dfrac{\text{g}H^*}{V^*}V^*_x(x) & &  2\dfrac{\text{g}H^*}{V^{*2}}V^*_x(x) \epm.
\end{equation}
The eigenvalues of the matrix $A(x)$ are
\begin{equation}
\lambda_1(x) = V^* + \sqrt{gH^*(x)} \hd \text{and} \hd - \lambda_2(x) = V^* - \sqrt{gH^*(x)}.
\end{equation}
Using these eigenvalues in the matrix $N(x)$ defined in \eqref{N(x)}, the next step is to select the functions $p_1(x)$ and $p_2(x)$ of the matrix $P(x)$ defined in \eqref{P(x)} to build the Lyapunov function candidate \eqref{basicLyap}.

In this application we shall see that it is sufficient to take $p_1 = p_2$ = constant. With $p_1 = p_2 = \demi$, the matrix $P(x)$ is 
\begin{equation}
P(x) = \bpm 	\text{g}H^*(x) + V^{*2}(x) & & - V^*(x) \\[0.5em] -V^*(x) & & 1 \epm.
\end{equation}
It is readily checked that, for all $x \in [0,L]$, this  matrix is positive definite (since $\det P(x) = \text{g}H^*(x)$). It follows that\footnote{From now on, when it does not lead to confusion, we often drop the argument $x$ to simplify the notations.}
\begin{equation} \label{matrixM}
M(x) = P(x) A(x) = 	\bpm -(\text{g}H^* - V^{*2})V^* & & \text{g}H^* - V^{*2} \\[0.5em] \text{g}H^* - V^{*2} & & V^* \epm.
\end{equation}
Moreover, we have
\begin{equation}
	P(x)B(x) + B^T(x)P(x) = \bpm 6 gH^* V^*_x & & -5 \dfrac{gH^*}{V^*}V^*_x \\[1em] -5 \dfrac{gH^*}{V^*}V^*_x & & 4 \dfrac{gH^*}{V^{*2}}V^*_x \epm,
\end{equation}
while the matrix $-M_x(x)$ is as follows:
\begin{equation}
	-M_x(x) = \bpm - 3V^{*2}V^*_x & & \dfrac{\text{g}H^* + 2 V^{*2}}{V^*} V^*_x\\[1em] \dfrac{\text{g}H^* + 2 V^{*2}}{V^*} V^*_x  & & - V^*_x \epm.
\end{equation}
Then we have
\begin{multline} \label{matrixMPB}
	-M_x(x) + P(x)B(x) + B^T(x)P(x) = \\[1em] \bpm (6\text{g}H^* - 3 V^{*2})V^*_x & & \dfrac{-4\text{g}H^* +2 V^{*2}}{V^*}V^*_x \\[1em] \dfrac{-4\text{g}H^* +2 V^{*2}}{V^*}V^*_x & & \left(\dfrac{4\text{g}H^* - V^{*2}}{V^{*2}} \right)V^*_x \epm\!.
\end{multline}
Under the subcritical condition \eqref{subcritical2}, this matrix is positive definite for all $x \in [0,L]$ because $V^*_x > 0$ and the determinant is
\begin{equation}
	\dfrac{V^*_x}{V^{*2}}\Big[4(\text{g}H^*)^2 + V^{*2}(2\text{g}H^* - V^{*2}) + 4\text{g}H^*(\text{g}H^* - V^{*2})\Big] > 0.
\end{equation}
Therefore the stability condition (b) is satisfied. Let us now address the boundary stability conditions relative to the boundaries. For that purpose, from \eqref{bcsv}, \eqref{bcab}, \eqref{bcfeedcont2}, we derive the boundary conditions of the linear error system which are
\begin{equation} \label{bcsvlin}
\begin{split}
&\tilde q(t,0) = 0 \hd (\text{i.e.} \; \alpha_h = 0, \alpha_q \neq 0),\\[0.5em]
&\tilde q(t,L) = \beta_L \tilde h(t,L) \; \text{ with } \; \beta_L = \frac{3}{2} \Big(c_g^2 Q^* \Big)^{\!\! 1/3} \hd (\text{i.e.} \; \beta_q \neq 0, \beta_h = - \beta_L \beta_q).\\[0.5em]
\end{split}
\end{equation}
We can check that the stability conditions (a1), (a2) and (a3) are verified. Indeed we have:
\begin{align*}
	\text{(a1)} \hd &\Longleftrightarrow \hd  \dfrac{\lambda_2(0)}{\lambda_1(0)} = \dfrac{\sqrt{gH^*(0)} - V^*(0)}{\sqrt{gH^*(0)} + V^*(0)} < 1. \\[1em]
	\text{(a2)} \hd &\Longleftrightarrow \hd \left( \dfrac{\lambda_1(L) - \beta_L}{\lambda_2(L) + \beta_L}\right)^{\!\! 2} <  \dfrac{\lambda_1(L)}{\lambda_2(L)} \\[0.5em] &\Longleftrightarrow \hd \left(\dfrac{\sqrt{gH^*(L)} + V^*(L) - \beta_L}{\sqrt{gH^*(L)} - V^*(L) + \beta_L}\right)^{\!\! 2} < \dfrac{\sqrt{gH^*(L)} + V^*(L)}{\sqrt{gH^*(L)} - V^*(L)} \\[0.5em] &\Longleftrightarrow \hd - V^*(L) \beta^2_L - \big(gH^*(L) - V^{*2}(L)\big)(2 \beta_L - V^*(L)) > 0. \\[0.5em] &\text{This inequality is satisfied if} \;\; 2\beta_L  > V^*(L). \\[1em] \text{(a3)} \hd &\Longleftrightarrow \hd \dfrac{\lambda_2(L)}{\lambda_1(L)} = \dfrac{\sqrt{gH^*(L)} - V^*(L)}{\sqrt{gH^*(L)} + V^*(L)} < 1.
\end{align*}
Furthermore, in this special case of Saint-Venant equations, it is worth noting that conditions (a1), (a2), (a3) and (b) ensure exponential stability not only in $L^2$ but also in $C^0$ for the linearized system (and thus locally in $C^1$ for the nonlinear system). This property follows directly from Theorem 3.2 and Corollary 1 in \cite{Hay19} (see also \cite{Hay19b}). This means that, if the disturbance $d(t)$ is bounded, then all the internal signals of the control system are also guaranteed to be bounded.

In this example of an open channel, we thus see that the feedforward control can  completely remove the effect of the disturbance while maintaining the system stability. This feedforward control analysis can be extended to the case of a channel with a space varying slope by using the Lyapunov function proposed for instance in \cite{Hay19c} and \cite{HaySha19}. For completeness, let us mention that feedforward control may also be used for the tracking of a time varying reference signal, a topic which is treated in \cite{RabMegLit09} using a parabolic partial differential equation resulting from a simplification of the Saint-Venant equations.

\begin{figure}[h!]
\begin{center}
\includegraphics[width=10cm]{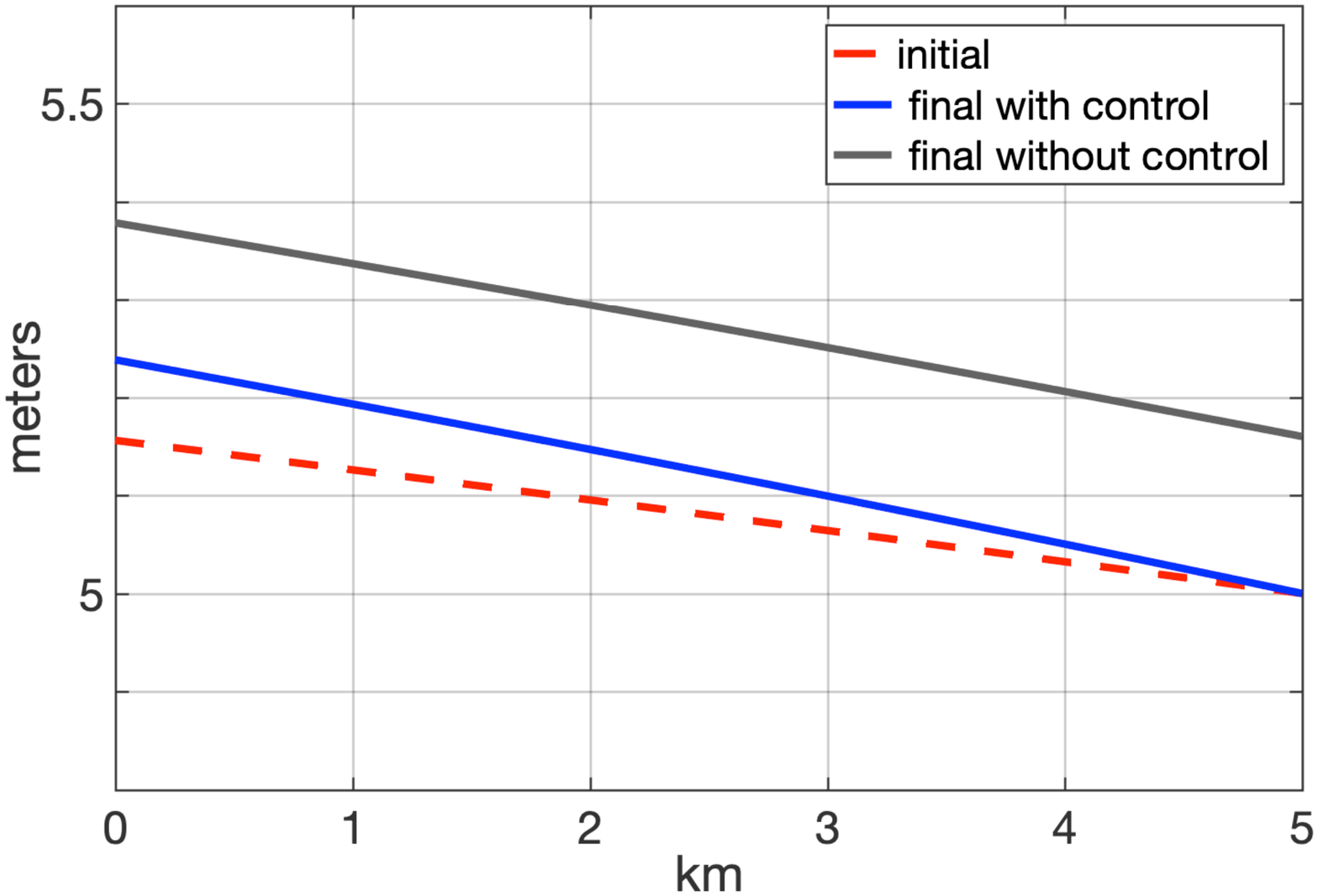}
\vspace{-0.6cm}
\caption{Steady state profiles of water level}
\label{profileH}
\end{center}
\end{figure}

Let us now illustrate this feedforward control design with a numerical simulation. The simulation is done with the `hpde' solver \cite{Sha05}. We consider a pool with the following parameters:
\begin{align*}
&\text{length:} \;\; L = 5000 \text{ (meters)}, \\
&\text{friction coefficient:} \;\; c_f = 0.01,	\\
&\text{discharge coefficient:} \;\; c_g = 2 \text{ m}^{4/3} \text{ s}^{-1}. 
\end{align*}

At the initial time ($t=0$), the system is at steady state with a constant flow rate per unit of width $Q^* = 2$ m$^3/$s and a boundary water level $H(0,L) = 5$ m. The initial steady state profile $H^*(x)$ of the water level is shown in Figure \ref{profileH} (dotted red curve).

\begin{figure}[h!]
\begin{center}
\includegraphics[width=10cm]{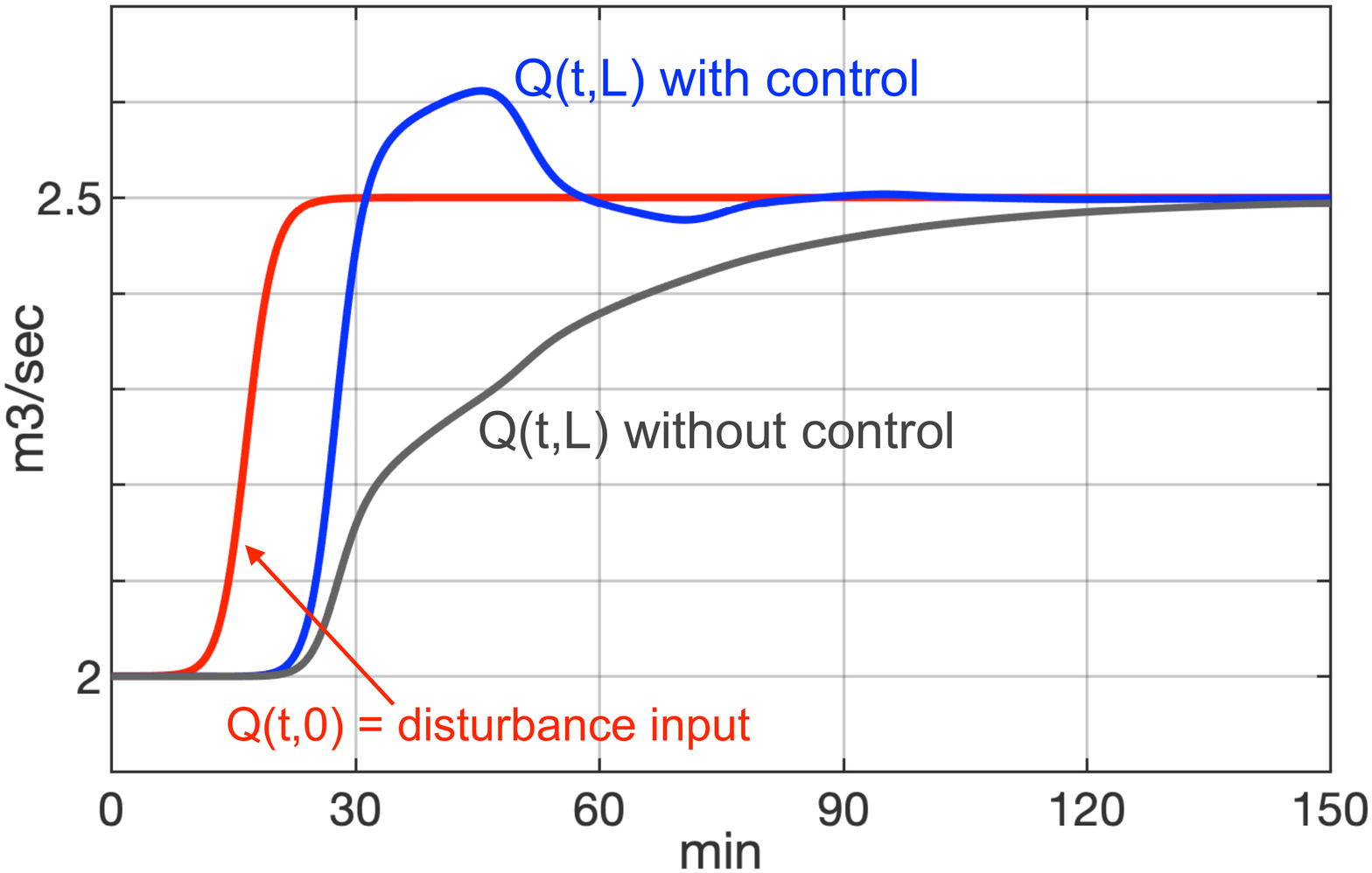}
\vspace{-0.6cm}
\caption{Input and output water flow rates per unit of width}
\label{qplot}
\end{center}
\end{figure}
\vspace{-1cm}
\begin{figure}[h!]
\begin{center}
\includegraphics[width=10cm]{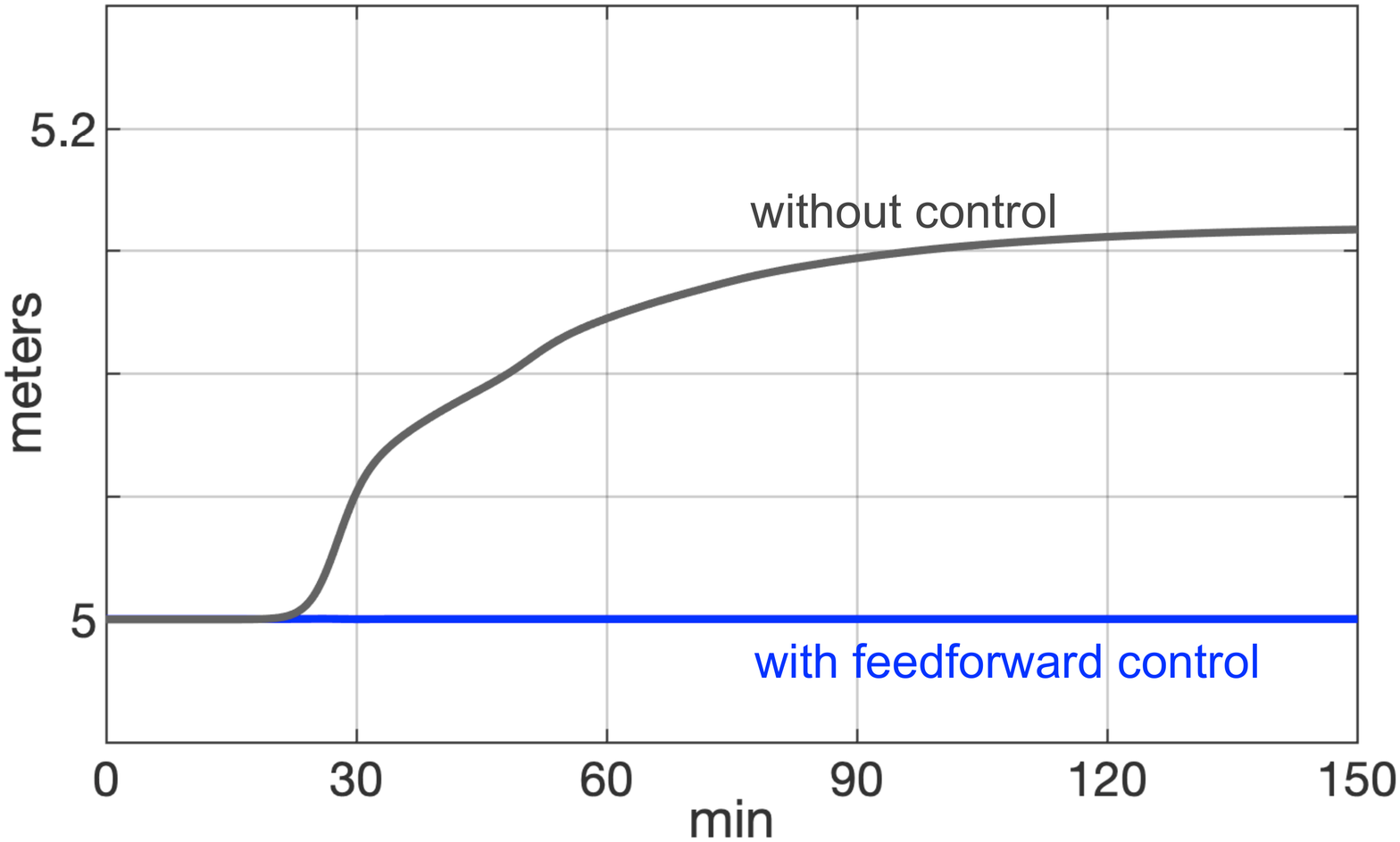}
\vspace{-0.6cm}
\caption{Time evolution of output water level $H(t,L)$}
\label{hplot}
\end{center}
\end{figure}

The system is subject to an isolated input disturbance which occurs around $t = 15$ minutes and is shown in Figure \ref{qplot}. The inflow rate $Q(t,0)$ is increased by about 25 \%, from 2 to 2.5 m$^3/$s (red curve). In this figure, we can also see the time evolution of the output flow rate $Q(t,L)$ with and without the feedforward control.

The control result is shown in Figure \ref{hplot}. With the feedforward control we see that the water level $H(t,L)$ (blue curve) remains constant at the set point $H^*_L = 5$ m despite the inflow disturbance. In contrast, without control, the same disturbance leads to an output level increase of about 16 cm (grey curve). The final steady state profile reached after the passage of the disturbance is  illustrated in Figure \ref{profileH}. 

Finally, let us also compute the parameter $\beta_L$ for this example. For the initial steady-state, we have:
\begin{equation}
Q^* = 2 \text{ m}^3\text{/s}	, \;\; H^*_L = 5 \text{ m}, \;\; V^*_L = 0.4 \text{ m/s} \; \; \text{and} \;\; \beta_L = \frac{3}{2} \Big(c_g^2 Q^* \Big)^{\!\! 1/3} = 3\,. 
\end{equation}
For the final steady-state, we have:
\begin{equation}
Q^* = 2.5 \text{ m}^3\text{/s}, \;\; H^*_L = 5 \text{ m}, \;\; V^*_L = 0.5 \text{ m/s} \; \; \text{and} \;\; \beta_L = \frac{3}{2} \Big(c_g^2 Q^* \Big)^{\!\! 1/3} = 3.23\,. 
\end{equation}
In both cases, we see that the stability condition $2\beta_L > V^*_L$ holds.

%\newpage
\section{Feedforward control in navigable rivers}

In navigable rivers the water is transported along the channel under the power of gravity through successive pools separated by control gates used for the control of the water level as illustrated in Figure \ref{river}. 
\begin{figure}[h!]
\begin{center}
\includegraphics[width=12cm]{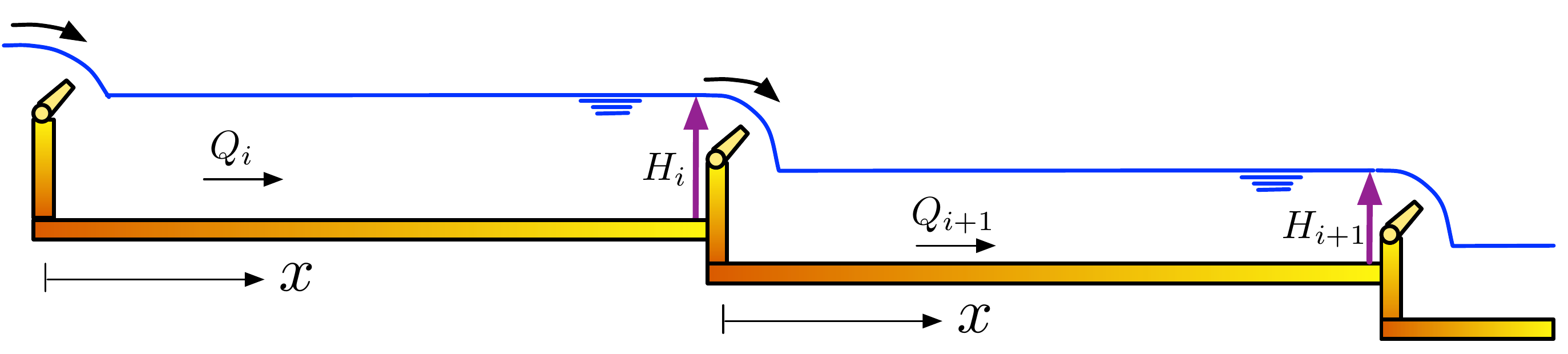}
\caption{Navigable river}
\label{river}
\end{center}
\end{figure}

In this section, for simplicity, we consider the ideal case of string of $n$ identical pools having the same length $L$ and the same rectangular cross section with a unit width. The channel dynamics are described by the following set of Saint-Venant equations
\begin{equation} \label{svriver}
\begin{matrix}
\p_t H_i + \p_x Q_i = 0,  \\[0.5em]
\p_t Q_i + \p_x \left(\dfrac {Q_i^2}{H_i} + \text{g} \dfrac{H_i^2}{2}\right)+ c_f Q_i^2= 0. 
\end{matrix} \hh i=1, \dots , n,
\end{equation}
and the following set of boundary conditions
\begin{equation} \label{bcriver}
\begin{matrix}
Q_1(t,0) = c_g\sqrt{\big[D_{\!\text{o}}(t)\big]^3}, \\[0.5em] Q_i(t,0) = Q_{i-1}(t,L) \hd i=2, \dots , n,\\[0.5em]
Q_i(t,L) = c_g\sqrt{\big[H_i(t,L)-U_i(t)\big]^3} \hd i=1, \dots , n, 
\end{matrix}
\end{equation}
where $H_i$ and $Q_i$ denote the water level and the flow rate in the $i$th pool, $U_i$ is the position of the $i$-th gate which is used as control action,  $c_f$ and $c_g$ are constant friction and gate shape coefficients respectively, $D_{\!\text{o}}(t)$ is the water head above the input gate considered here as the external measurable disturbance. Furthermore it is assumed that the water levels $H_i(t,L)$ are measurable at the downstream side of the pools.

Let us now assume that the objective is to find a set of  \textit{feedforward control laws} $U_i(t)$, function of the measured disturbance $D_{\!\text{o}}(t)$ and the measurable levels $H_i(t,L)$, such that each output $Y_i(t) = H_i(t,L) -H^*_L$ is identically zero.

In this framework, each pool can be considered as a dynamical system with a controlled output $Y_i(t) = H_i(t,L) - H^*_L$ and a disturbance input $D_{i-1}(t)$ = $Q_{i-1}(t,L)$. It is then natural to design the feedforward control laws $U_i$ for each pool on the pattern of the control which was derived in the previous section for a single pool, as follows:
\begin{subequations} \label{ffcontrolriver} \begin{align} 
&\text{For } i=1,\dots , n, \nonumber \\[0.5em]
&\p_t \widehat H_i + \p_x \widehat Q_i = 0,  \\[0.5em]
&\p_t \widehat Q_i + \p_x \left(\dfrac{\widehat Q_i^2}{\widehat H_i} + \text{g} \dfrac{\widehat H_i^2}{2}\right) + c_f \dfrac{\widehat Q_i^2}{\widehat H_i^2} = 0, \\[0.5em]
&\widehat H_i(t,L) = H^*_L,\\[0.5em]
&\widehat Q_i(t,0) = c_g\sqrt{\big[D_{i-1}(t)\big]^3}, \\[0.5em]
&D_i(t) = H_i(t,L)-U_i(t),	\\[0.5em]
&U_i(t) = H^*_L - \big(c_g^{-1} \widehat Q_i(t,L)\big)^{2/3}.
\end{align}\end{subequations}
\begin{theorem}
For the control system \er{svriver}, \er{bcriver}, \er{ffcontrolriver}, for all $i = 1, \dots , n$, assume that  the initial conditions satisfy
\begin{equation} \label{initriver}
H_i(0,x) = \widehat H_i(0,x), \;\; Q_i(0,x) = \widehat Q_i(0,x), \;\; \forall x \in [0,L], 
\end{equation}
with $H_i(0,L) = \widehat H_i(0,L) = H^*_L$.
Then, for all positive $t$ and for all $i = 1, \dots , n$, it holds that $Y_i(t) = H_i(t,L) - H^*_L = 0$. \qed
\end{theorem}
\noindent The proof of this theorem is clearly an immediate consequence of the proof of Theorem \ref{theorem1}.

In order to discuss the system stability we introduce the following notations:
\begin{equation}
\begin{split}
&\text{For } i=1,\dots , n, \\[0.5em]
	&h_i(t,x) = H_i(t,x) - H^*(x), \hd q_i(t,x) = Q_i(t,x) - Q^*, \\[0.5em]
	&\hat h_i(t,x) = \widehat H_i(t,x) - H^*(x), \hd \hat q_i(t,x) = \widehat Q_i(t,x) - Q^*, \\[0.5em]
	& \tilde h_i(t,x) = h_i(t,x) - \hat h_i(t,x), \hd  \tilde q_i(t,x) = q_i(t,x) - \hat q_i(t,x),\\[0.5em]
	&\tilde z_i = \bpm \tilde h_i \\[0.5em] \tilde q_i \epm, \;\; \hat z_i = \bpm \hat h_i \\[0.5em] \hat q_i \epm.
	\end{split}
\end{equation}
The linear error system is written
\begin{equation} \label{errorlinriver}
\p_t \tilde z_i + A(x)\p_x \tilde z_i + B(x) \tilde z_i = 0, \hh i= 1, ... , n,
\end{equation}
with the boundary conditions
\begin{equation} \label{bcerrorlinriver}
\begin{split}
&\tilde q_i(t,0) = 0	, \\[0.5em]
&\tilde q_i(t,L) = \beta_L \tilde h_i(t,L),
\end{split}
\hh i= 1, ... , n.
\end{equation}
Here we see that the error subsystems corresponding to each pool are decoupled. Therefore the stability of the global error system \eqref{errorlinriver}, \eqref{bcerrorlinriver} directly results from the stability of the error system for a single pool which was established in Section \ref{sectionchannel}.

On the other hand, the linearization of the controller system is written
\begin{equation} \label{controllinriver}
\p_t \hat z_i + A(x)\p_x \hat z_i + B(x) \hat z_i = 0, \hh i= 1, ... , n,
\end{equation}
with the boundary conditions
\begin{equation} \label{bccontrollinriver}
\begin{split}
&\hat q_1(t,0) = \alpha_0 d(t), \\[0.5em]
&\hat q_{i}(t,L)  = \hat q_{i+1}(t,0) - \tilde q_i(t,L), \hd i= 1, ... , n-1, \\[0.5em]
&\hat h_i(t,L) = 0, \hd i= 1, ... , n. 
\end{split}
\end{equation}
Here, we remark that the subsystems corresponding to each pool are interconnected through the boundary conditions and their respective stabilities cannot be considered separately. Therefore we introduce the following Lyapunov function candidate
\begin{equation}
\mb{V} = \sum_{i=1}^n \int_0^L  \omega_i\big(\hat z_i^T P(x) \hat z_i\big) dx
\end{equation}
where $\omega_i$ are positive coefficients to be determined.

The time derivative of this Lyapunov function along the system solutions is then as follows:
\begin{equation}
\dfrac{d\mb{V}}{dt} = {\cal I}(t) + {\cal B}(t)	
\end{equation}
with the ``internal'' term
\begin{equation} \label{dVdtlin}
{\cal I}(t) = - \sum_{i=1}^n \omega_i \int_0^L \hat z_i^T \Big( -M_x(x) + B^T(x)P(x) + P(x) B(x) \Big) \hat z_i \, dx  
\end{equation}
and the ``boundary'' term
\begin{equation}
{\cal B}(t) = - \sum_{i=1}^n \omega_i \Big[ z_i^T M(x) z_i \Big]_0^L .
\end{equation}
The matrices $M(x)$ and $-M_x(x) + B^T(x)P(x) + P(x) B(x)$ are those defined in the previous section by equations \eqref{matrixM} and \eqref{matrixMPB} respectively. It follows directly that the interior term is negative : ${\cal I}(t) < 0$. Moreover, using the definition of $M(x)$ and the boundary conditions \eqref{bccontrollinriver}, the boundary term ${\cal B}(t)$ may be written as follows:
\begin{equation}
\begin{split}
{\cal B}(t) =  &- \omega_1	V^*(0)(gH^*(0) - V^{*2}(0)) h_1^2(t,0) - \omega_n V^*(L) \hat q_n^2(t,L)
\\[0.5em]
&- \sum_{i=1}^{n-1} \bpm \hat q_i(t,L) & \hat h_{i+1}(t,0) \epm  \Omega_i \bpm  \hat q_i(t,L) \\[0.5em]  \hat h_{i+1}(t,0) \epm \\[0.5em] 
&+ 2 \,\omega_1\alpha_0(gH^*(0) - V^{*2}(0))  \hat h_1(t,0) d(t) + \omega_1 \alpha_0^2 V^*(0) d^2(t) \\[0.5em]
&+ \sum_{i=1}^{n-1} \omega_{i+1} V^*(0) \Big[ -2 \hat q_i(t,L) \tilde q_i(t,L) + V^*(0) \tilde q_i^2(t,L) \Big]
\end{split}
\end{equation}
where for $i = 1, \dots ,n-1$
\begin{equation}
\Omega_i =  \bpm \omega_i V^*(L) - \omega_{i+1}V^*(0) & \hd - \omega_{i+1}(gH^*(0) - V^{*2}(0)) \\[1em] - \omega_{i+1}(gH^*(0) - V^{*2}(0)) & \hd \omega_{i+1}(gH^*(0) - V^{*2}(0))V^*(0) \epm.	
\end{equation} 
Assume that the positive coefficients $\omega_i$ are selected according to
\begin{equation}
\dfrac{\omega_{i+1}}{\omega_i} = \varepsilon 
\end{equation}
where $\varepsilon$ is a  positive constant to be determined. Then
\begin{equation}
\Omega_i =  \omega_i \bpm V^*(L) - \varepsilon V^*(0) & \hd - \varepsilon (gH^*(0) - V^{*2}(0)) \\[1em] - \varepsilon (gH^*(0) - V^{*2}(0)) & \hd \varepsilon (gH^*(0) - V^{*2}(0))V^*(0) \epm.	
\end{equation} 
Using this expression, it can be seen that, under the subcritical condition \eqref{subcritical2}, each matrix $\Omega_i$ is positive definite provided $\varepsilon$ is selected such that
\begin{equation}
	\varepsilon < \dfrac{V^*(0)V^*(L)}{gH^*(0)}.
\end{equation}
Hence it follows that the (linearized) controller system \eqref{controllinriver}, \eqref{bccontrollinriver}, with inputs $d(t)$ and $\tilde q_i(t,L)$, is $L^2$-input-to-state stable with an estimate of the form 
\begin{equation}
	\lVert \hat z(t,\cdot)\rVert_{L^{2}} \ls C_{1}\lVert \hat z(t,0)\rVert_{L^{2}}e^{-\nu t}+C_{2} \; \sup _{t \gs 0} \Big[ |d(t)| + \sum_{i=1}^{n-1} |\tilde q_i(t,L)|\Big] 
\end{equation}
where $\nu$ and $C_i$ $(i=1,2)$ are positive constants. Thus, here again, we can conclude that the state of the system is bounded and that the feedforward control objective is achieved. Let us now illustrate the control performance through simulation experiments.

\subsubsection*{Simulation experiments}

We consider a channel with two identical successive pools as shown in Figure \ref{river}, with the following parameters:
\begin{align*}
&\text{length:} \;\; L = 5000 \text{ (meters)}, \\
&\text{friction coefficient:} \;\; c_f = 0.008,	\\
&\text{discharge coefficient:} \;\; c_g = 2 \text{ m}^{4/3} \text{ s}^{-1}. 
\end{align*}
The results of a simulation of the feedforward control is given in Figure \ref{simfeedriver} with set points $H^*_L = 5$ meters in the first pool and $H^*_L = 4.9$ meters in the second pool. 
\begin{figure}[h]
  \centering
  \begin{subfigure}[b]{0.45\textwidth}
     \centering \includegraphics[width=5.8cm]{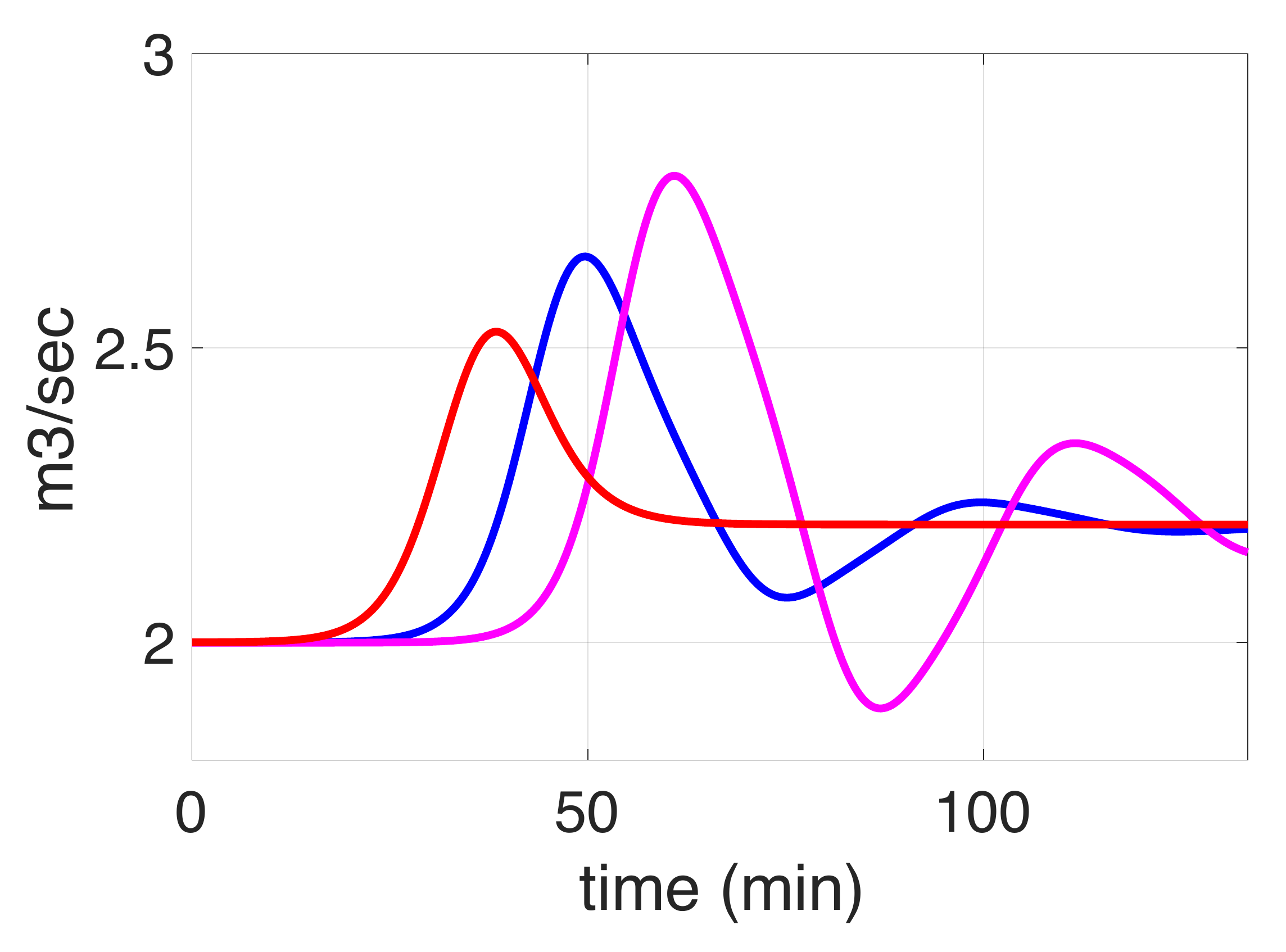}
     \caption{Flow rate $Q(t,L)$}
     \label{flow}
  \end{subfigure}
  \hspace{5pt}\vspace{0.6cm}
  \begin{subfigure}[b]{0.45\textwidth}
     \centering \includegraphics[width=5.5cm]{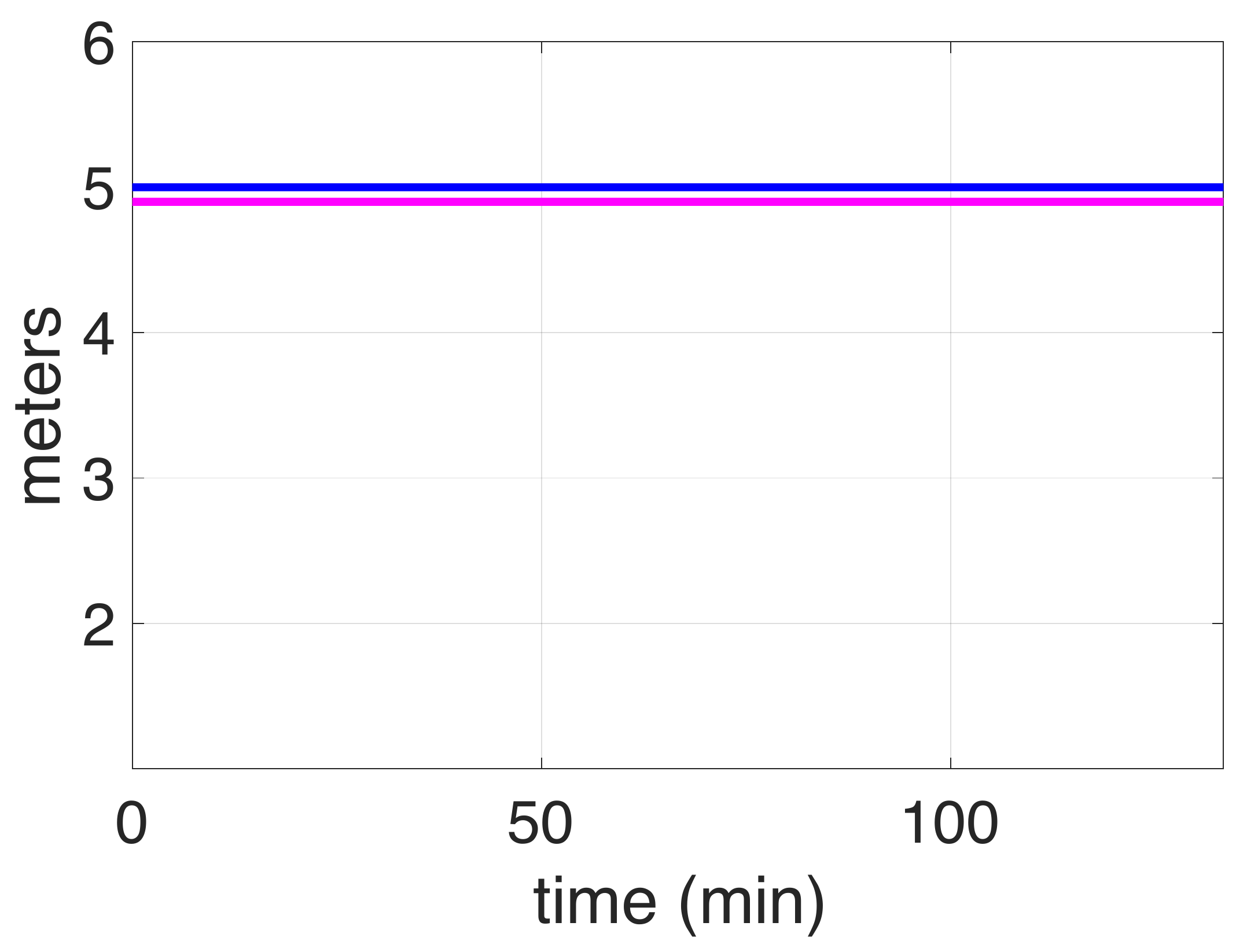}
     \caption{Water level $H(t,L)$}
     \label{level}
  \end{subfigure}
  \begin{subfigure}[b]{0.45\textwidth}
     \centering \includegraphics[width=5.8cm]{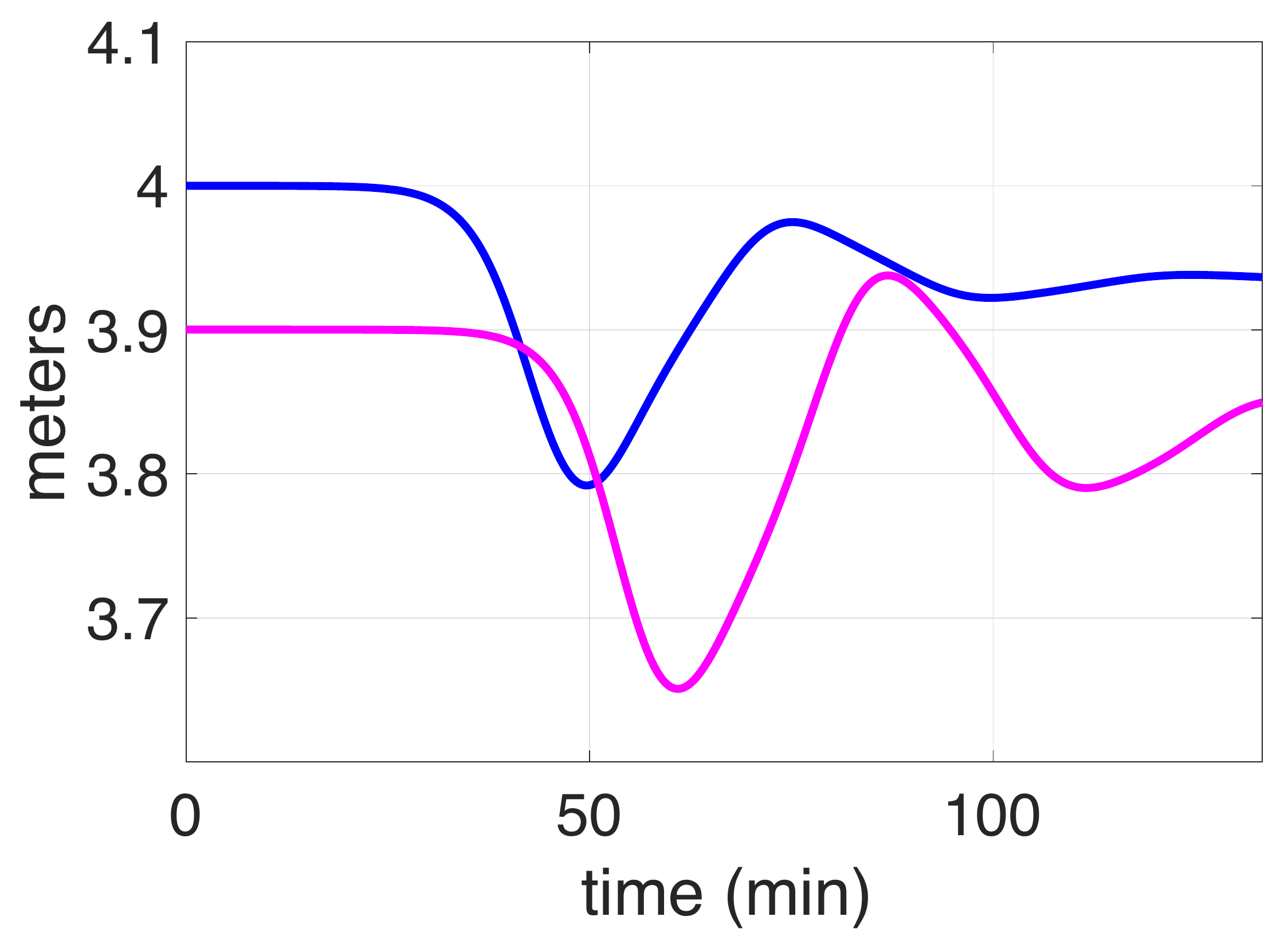}
     \caption{Control actions $U(t)$}
     \label{control}
  \end{subfigure}
  \caption{Simulation of feedforward control in a channel with two successive pools: {\color{red}$\scriptstyle \blacksquare$} = disturbance input, {\color{blue}$\scriptstyle \blacksquare$} = first pool, {\color{magenta}$\scriptstyle \blacksquare$} = second pool.}
  \label{simfeedriver}
\end{figure}

At the initial time ($t=0$), the system is at steady state with a constant flow rate per unit of width $Q^* = 2$ m$^3/$s and boundary water levels $H_1(0,L) = 5$ m and $H_2(0,L) = 4.9$ m respectively. 
The system is subject to an input disturbance which occurs around $t = 15$ minutes and is shown in Figure \ref{flow} (red curve). This disturbance takes the form of a pulse starting from the steady-state value 2 m$^3/$s, then peaking at 2.5 m$^3/$s (i.e. an increase of about 25 \%), and finally stabilizing at 2.2  m$^3/$s. In this figure, we can also see the time evolution of the flow rates $Q(t,L)$ in the two pools under the  feedforward control (blue and magenta curves). The control actions computed by the two feedforward controllers are shown in Figure \ref{control}. Obviously, as expected, we can see in Figure \ref{level} that the feedforward control is perfectly efficient and that the water levels $H(t,L)$ in the two pools are totally insensitive to the disturbance.

The simulation is done with a friction coefficient value which is in the range of usual values for natural channels and rivers. In this case, there is however a drawback, that is very visible in Figure \ref{flow}, under the form of an amplification of the oscillations of the flow rates in the downstream direction. This can be detrimental in some practical applications. In order to mitigate this phenomenon, some filtering of the control must be applied. 

A simple very natural and efficient way to implement such filtering is to fictitiously increase the value of the friction coefficient in the feedforward controller. This strategy is illustrated in Figure \ref{simfeedriver2} where the control laws \eqref{ffcontrolriver} are implemented with a fake overestimated value of the friction $\hat c_f = 0.024$. The nice performance of this control can be appreciated in Figures \ref{flow2} and \ref{level2}. Indeed, it can be observed that, in this case, the flow rates are no longer amplified in the downstream direction while the water levels remain nevertheless rather insensitive to the disturbance effect.
\begin{figure}[htp]
  \centering 
  \begin{subfigure}[b]{0.45\textwidth}
     \centering \includegraphics[width=6.0cm]{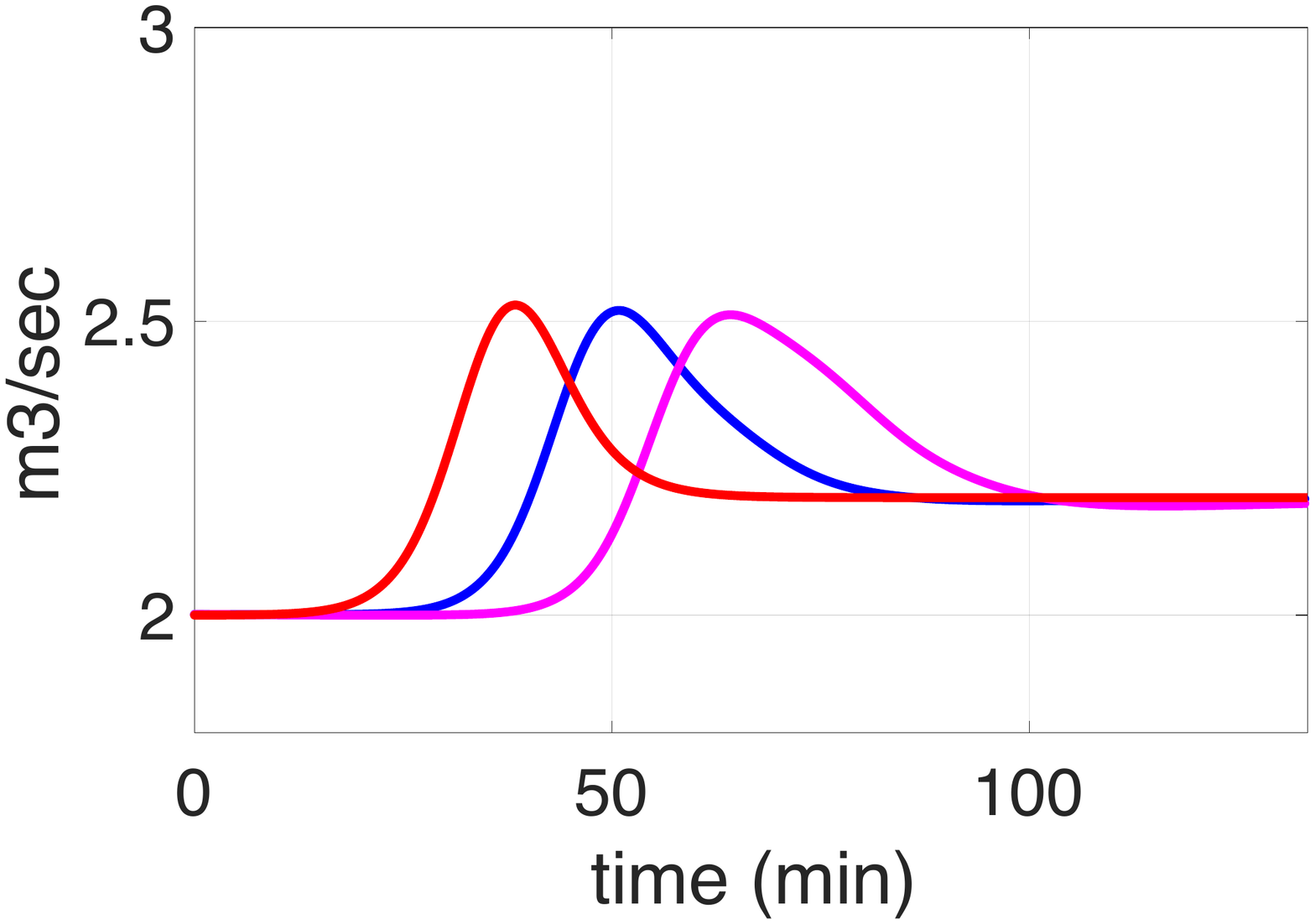}
     \caption{Flow rate $Q(t,L)$}
     \label{flow2}
  \end{subfigure}
  \hspace{5pt}\vspace{0.6cm}
  \begin{subfigure}[b]{0.45\textwidth}
     \centering \includegraphics[width=5.7cm]{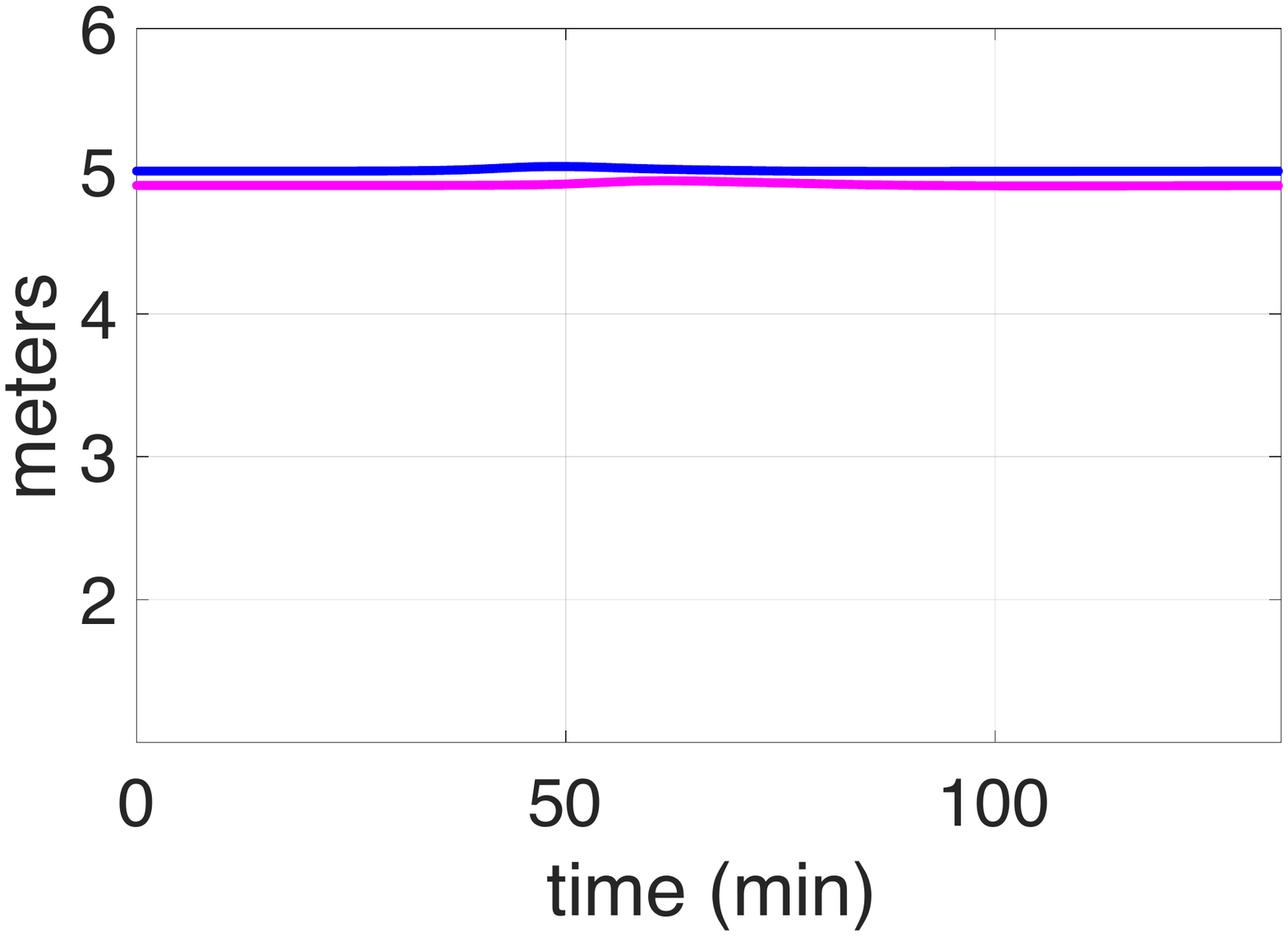}
     \caption{Water level $H(t,L)$}
     \label{level2}
  \end{subfigure}
  \begin{subfigure}[b]{0.45\textwidth}
     \centering \includegraphics[width=6.0cm]{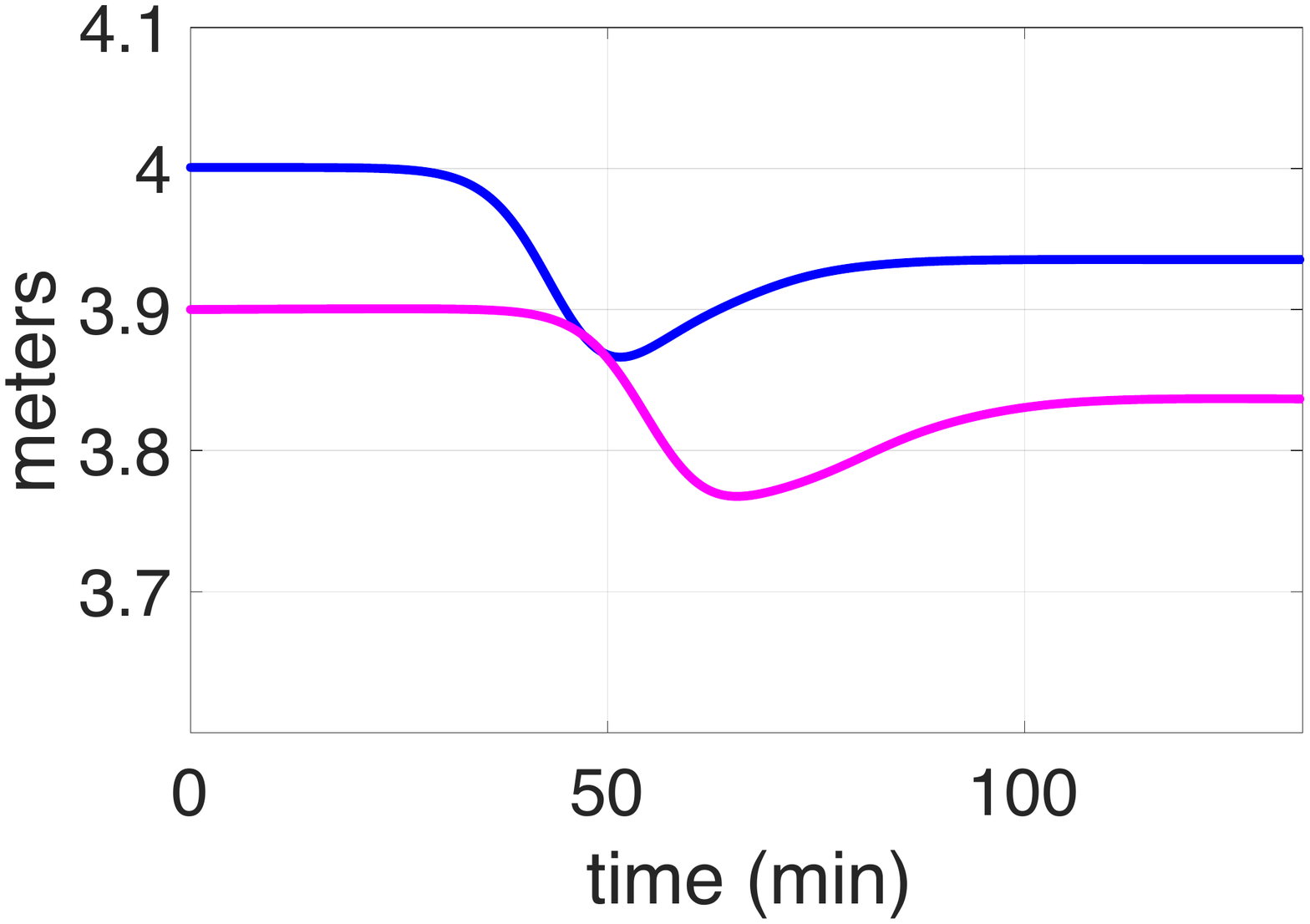}
     \caption{Control actions $U(t)$}
     \label{control2}
  \end{subfigure}
  \caption{Reduced amplification of the flow oscillations with a modified feedforward controller: {\color{red}$\scriptstyle \blacksquare$} = disturbance input, {\color{blue}$\scriptstyle \blacksquare$} = first pool, {\color{magenta}$\scriptstyle \blacksquare$} = second pool.}
  \label{simfeedriver2}
\end{figure}

\section{Conclusions}
In this paper, we have addressed the design of feedforward controllers for a general class of $2 \times 2$ hyperbolic systems with a disturbance input located at one boundary and a control actuation at the other boundary. The goal is to design a feedforward control that makes the system output insensitive to a measured disturbance input. 

 The problem was first stated and studied in the frequency domain for a simple linear system. Then, our main contribution was to extend the theory, in the time domain, to general nonlinear hyperbolic systems. First it has been shown that there exists an ideal causal feedforward dynamic controller that achieves perfect control. In a second step, sufficient conditions have been given under which the controller, in addition to being causal, ensures the stability of the overall control system. 
 
 The method has been illustrated with an application to the control of an open channel represented by Saint-Venant equations where the objective is to make the output water level insensitive to the variations of the input flow rate. In the last section, we have discussed a more complex application to a cascade of pools where a blind application of perfect feedforward control can lead to detrimental oscillations. A pragmatic way of modifying the control law to solve this problem has been proposed and validated with a simulation experiment. 	
 
 Finally, we would also like to mention that the application to the Saint-Venant equations with hydraulic gates can be transposed to gas pipelines with compressors described by the isentropic Euler equations. The interested reader can consult the references \cite{GugHer09}, \cite{DicGugLeu10}, \cite{GugLeuTam12} and \cite{BasCor17}.

\section*{Acknowledgments}
This research was supported by the ANR project Finite4SoS (No.ANR 15-CE23-0007), INRIA team CAGE, the NSF CPS Synergy project 
"Smoothing Traffic via Energy-efficient Autonomous Driving" (STEAD)
CNS 1837481 and the French Corps des IPEF.


\begin{thebibliography}{10}

\bibitem{Aam13}
O-M. Aamo.
\newblock Disturbance rejection in {2 $\times$ 2} linear hyperbolic systems.
\newblock {\em IEEE Transactions on Automatic Control}, 58(5):1095--1106, May
  2013.

\bibitem{AnfAam17b}
H.~Anfinsen and O-M. Aamo.
\newblock Disturbance rejection in general heterodirectional 1-d linear
  hyperbolic systems using collocated sensing and control.
\newblock {\em Automatica}, 76:230--242, 2017.

\bibitem{BasCor14}
G.~Bastin and J-M. Coron.
\newblock {\em Stability and Boundary Stabilisation of 1-D Hyperbolic Systems}.
\newblock Number~88 in Progress in Nonlinear Differential Equations and Their
  Applications. Springer International, 2016.

\bibitem{BasCor17}
G.~Bastin and J-M. Coron.
\newblock A quadratic {Lyapunov} function for hyperbolic density--velocity
  systems with nonuniform steady states.
\newblock {\em Systems and Control Letters}, 104:66--71, 2017.

\bibitem{BasCordAn05}
G.~Bastin, J-M. Coron, B.~d'Andr\'ea{-}Novel, and L.~Moens.
\newblock Boundary control for exact cancellation of boundary disturbances in
  hyperbolic systems of conservation laws.
\newblock In {\em Proceedings 44th IEEE Conference on Decision and Control and
  the European Control Conference 2005}, pages 1086--1089, Seville, Spain,
  December 12-15 2005.

\bibitem{BasCorHay20}
G.~Bastin, J-M. Coron, and A.~Hayat.
\newblock Input-to-state stability in sup norms for hyperbolic systems with
  boundary disturbances.
\newblock {\em Preprint}, 2020.

\bibitem{Deu17b}
J.~Deutscher.
\newblock Output regulation for general linear heterodirectional hyperbolic
  systems with spatially-varying coefficients.
\newblock {\em Automatica}, 85:34--42, 2017.

\bibitem{Deu17a}
J.~Deutscher.
\newblock Robust output regulation by observer-based feedforward control.
\newblock {\em International Journal of Systems Science}, 48(4):795--804, 2017.

\bibitem{DeuGab19}
J.~Deutscher and J.~Gabriel.
\newblock Periodic output regulation for general linear heterodirectional
  hyperbolic systems.
\newblock {\em Automatica}, 103:208--216, 2019.

\bibitem{DicGugLeu10}
M.~Dick, M.~Gugat, and G.~Leugering.
\newblock Classical solutions and feedback stabilisation for the gas flow in a
  sequence of pipes.
\newblock {\em Networks and Heterogeneous Media}, 5(4):691--709, 2010.

\bibitem{FerPri20}
F.~Ferrante and C.~Prieur.
\newblock Boundary control design for conservation laws in the presence of
  measurement noise.
\newblock 2020.
\newblock Preprint.

\bibitem{Gla19}
T.~Glad.
\newblock {\em Lecture notes on control of nonlinear systems - Chapter 11}.
\newblock {Link{\"o}pings} University -
  {https://www.control.isy.liu.se/student/graduate/nonlin/ }, 2019.

\bibitem{GugHer09}
M.~Gugat and M.~Herty.
\newblock Existence of classical solutions and feedback stabilisation for the
  flow in gas networks.
\newblock {\em ESAIM Control Optimisation and Calculus of Variations},
  17(1):28--51, 2011.

\bibitem{GugLeuTam12}
M.~Gugat, G.~Leugering, S.~Tamasoiu, and K.~Wang.
\newblock ${H}^2$-stabilization of the isothermal euler equations: a {L}yapunov
  function approach.
\newblock {\em Chinese Annals of Mathematics. Series B}, 33(4):479--500, 2012.

\bibitem{GuzHag11}
J.L. Guzman and T.~Hagglund.
\newblock Simple tuning rules for feedforward compensators.
\newblock {\em Journal of Process Control}, 21:92--102, 2011.

\bibitem{HasHag12}
M.~Hast and T.~Hagglund.
\newblock Design of optimal low-order feedforward controllers.
\newblock In {\em Proceedings 2nd IFAC Conference on Advances in PID Control},
  volume~45 of {\em IFAC Proceedings Volumes}, pages 483--488, 2012.

\bibitem{Hay19b}
A.~Hayat.
\newblock Exponential stability of general {1-D} quasilinear systems with
  source terms for the {$C_1$} norm under boundary conditions.
\newblock {\em SIAM Journal of Control and Optimization}, 57(6):3603--3665,
  2019.

\bibitem{Hay19}
A.~Hayat.
\newblock On boundary stability of inhomogeneous $2 \times 2$ {1-D} hyperbolic
  systems for the {$C_1$} norm.
\newblock {\em ESAIM: COCV}, 25:article 82, 2019.

\bibitem{Hay19c}
A.~Hayat.
\newblock {PI} controller for the general {S}aint-{V}enant equations.
\newblock 2019.
\newblock Preprint, hal-01827988.

\bibitem{HaySha19}
A.~Hayat and P.~Shang.
\newblock A quadratic {L}yapunov function for {S}aint-{V}enant equations with
  arbitrary friction and space-varying slope.
\newblock {\em Automatica}, 100:52--60, 2019.

\bibitem{HovBit09}
M.~Hovd and R.R. Bitmead.
\newblock Feedforward for stabilization.
\newblock In {\em Proceedings 7th IFAC Symposium on Advanced Control of
  Chemical Processes}, volume~42 of {\em IFAC Proceedings Volumes}, pages 602
  -- 606, 2009.

\bibitem{KarKrs19}
I.~Karafyllis and M.~Krstic.
\newblock {\em Input-to-state stability for {PDE}s}.
\newblock Communications and Control Engineering Series. Springer, Cham, 2019.

\bibitem{LitFro09c}
X.~Litrico and V.~Fromion.
\newblock {\em Modeling and Control of Hydrosystems. A Frequency Domain
  Approach.}
\newblock Springer Verlag, 2009.

\bibitem{LitFroSco07}
X.~Litrico, V.~Fromion, and G.~Scorletti.
\newblock Robust feedforward boundary control of hyperbolic conservation laws.
\newblock {\em Networks and Heterogeneous Media}, 2(4):715--729, 2007.

\bibitem{MorZaf89}
M.~Morari and E.~Zafiriou.
\newblock {\em Robust Process Control}.
\newblock Prentice Hall, 1989.

\bibitem{RabMegLit09}
T.~S. Rabbani, F.~{Di Meglio}, X.~Litrico, and A.M. Bayen.
\newblock Feed-forward control of open channel flow using differential
  flatness.
\newblock {\em IEEE Transactions on Control Systems Technology},
  18(1):213--221, 2010.

\bibitem{SebEdgMel11}
D.E. Seborg, T.F. Edgar, D.A Mellichamp, and F.J.~Doyle III.
\newblock {\em Process Dynamics and Control - 3rd Edition}.
\newblock Wiley, 2011.

\bibitem{Sha05}
L.F. Shampine.
\newblock Solving hyperbolic {PDE}s in {MATLAB}.
\newblock {\em Applied Numerical Analysis $\&$ Computational Mathematics},
  2(3):346--358, 2005.

\bibitem{Wan06}
Z.~Wang.
\newblock Exact controllability for nonautonomous first order quasilinear
  hyperbolic systems.
\newblock {\em Chinese Annals of Mathematics. Series B}, 27(6):643--656, 2006.

\bibitem{YorBan20}
G.Y. Weldegiyorgis and M.K. Banda.
\newblock {An analysis of the input-to-state stabilization of linear hyperbolic
  systems of balance laws with boundary disturbances}.
\newblock 2020.
\newblock Preprint.

\end{thebibliography}
\end{document}